\documentclass{article}
\usepackage{amsmath,amsfonts,amsthm}
\usepackage{xcolor,hyperref,upgreek}
\newcommand{\ii}{\mathrm{i}}
\newcommand{\C}{\mathbb{C}}
\newcommand{\R}{\mathbb{R}}
\newcommand{\su}{\mathfrak{su}}

\newcommand{\tr}{\operatorname{Tr}}
\newcommand{\g}{\mathfrak{g}}
\DeclareMathOperator{\End}{End}
\DeclareMathOperator{\Aut}{Aut}
\DeclareMathOperator{\Hom}{Hom}
\newcommand{\ad}{\operatorname{ad}}
\DeclareMathOperator{\coker}{coker}
\DeclareMathOperator{\cl}{cl}
\newtheorem{theorem}{Theorem}
\newtheorem{proposition}[theorem]{Proposition}
\newtheorem{lemma}[theorem]{Lemma}
\newtheorem{corollary}[theorem]{Corollary}
\newtheorem{conjecture}[theorem]{Conjecture}
\theoremstyle{definition}
\newtheorem{remark}[theorem]{Remark}
\title{$L^2$ geometry of hyperbolic monopoles}
\date{28th July 2025}
\author{Guido Franchetti${}^1$ and Derek Harland${}^2$
\medskip\\
{\normalsize ${}^1$Dipartimento di Matematica, University of Pisa, Italy}\\
{\normalsize guido.franchetti@unipi.it}
\medskip\\
{\normalsize ${}^2$School of Mathematics, University of Leeds, UK}\\
{\normalsize d.g.harland@leeds.ac.uk}
}
\begin{document}
\maketitle

\abstract{It is well-known that the $L^2$ metric on the moduli space of hyperbolic monopoles, defined using the Coulomb gauge-fixing condition, diverges.  This article shows that an alternative gauge-fixing condition inspired by supersymmetry cures this divergence.  The resulting geometry is a hyperbolic analogue of the hyperk\"ahler geometry of euclidean monopole moduli spaces.}

\section{Introduction}

\subsection{Background and motivation}
%Definition of monopoles
An $SU(2)$ monopole consists of an $SU(2)$ connection $A$ and an $\mathfrak{su}(2)$-valued section $\Phi$ over a riemannian 3-manifold solving the BPS equation
\begin{equation}\label{BPSintro}
d^A\Phi = \ast F^A
\end{equation}
and suitable boundary conditions.  The space of all solutions modulo gauge transformations is called the moduli space, and metrics on moduli spaces play a central role in the study of monopoles.  The $L^2$ metric on the moduli space of monopoles in euclidean space is defined by the integral
\begin{equation}\label{metricintro}
g = -\frac12\int \tr(a\wedge\ast a+\phi\wedge\ast\phi),
\end{equation}
in which $(a,\phi)$ are a Lie algebra-valued 1-form and 0-form which represent a tangent vector to the moduli space and which solve the linearisation of the BPS equation \eqref{BPSintro}.  A given tangent vector to the moduli space has many gauge-equivalent representatives $(a,\phi)$, so to make the metric \eqref{metricintro} well-defined one must impose a gauge-fixing condition.

For monopoles on euclidean space, the gauge-fixing condition
\begin{equation}\label{coulombintro}
\ast d^A\ast a + [\Phi,\phi]=0
\end{equation}
admits a unique solution and leads to a well-defined metric on the moduli space.  This metric was introduced by Manton \cite{manton:1982}, who proposed that the geodesics of this metric should approximate the dynamics of slowly-moving monopoles.  This conjecture was later confirmed by Stuart \cite{stuart:1994}.

The metric on the moduli space of euclidean monopoles is hyperk\"ahler, and this fact allowed Atiyah--Hitchin to calculate the metric on the moduli space of centred 2-monopoles explicitly \cite{atiyah:2014}.  The Atiyah--Hitchin metric is now one of the best-known examples of a hyperk\"ahler manifold.  Many other prominent examples of hyperk\"ahler metrics arise as moduli spaces of monopoles \cite{Lee:1996,cherkis:1999,cherkis:2002,cherkis:2012}.

The moduli space metric can be used to study not only classical but also quantum mechanics of monopoles \cite{gibbons:1986}.  Famously, Sen \cite{sen:1994} made predictions for the $L^2$ cohomology of centred monopole moduli spaces, and these conjectures have been confirmed for charges 2 \cite{hitchin:2000} and 3 \cite{kottke:2024}.

Monopoles on hyperbolic space share much mathematical structure with their euclidean counterparts, including correspondences with spectral curves, rational maps, and (discrete) Nahm equations, so it would be natural to expect their moduli spaces to admit interesting metrics.  However, it has been known for a long time that the gauge-fixing condition \eqref{coulombintro} leads to the integral \eqref{metricintro} being divergent \cite{braam:1989}.  Thus the construction of an $L^2$-metric on the moduli space of hyperbolic monopoles is a long-standing open problem.

In view of the difficulty with the $L^2$ metric, many alternative constructions of metrics and other geometric structures on hyperbolic monopole moduli spaces have been proposed.  Hitchin \cite{hitchin:1996} constructed an Einstein metric on the moduli space of inversion-symmetric charge 2 monopoles using twistor theory.  Unfortunately, it has proven difficult to generalise this approach to higher charges \cite{hitchin:2008}, and it is not clear what relevance this metric has to monopole dynamics.  More recently, Bielawski--Schwachhofer \cite{bielawski:2013,bielawski:2013b} introduced pluricomplex structures on the moduli spaces, building on twistor theory developed by Nash \cite{nash:2007}.  Figueroa-O'Farrill--Gharamti \cite{figueroa-ofarrill:2014} recovered these pluricomplex structures starting from an analysis of the supersymmetry of hyperbolic monopoles.  As with Hitchin's work, the relevance of Bielawski--Schwachhofer's approach to the classical dynamics of monopoles is unclear.

Gibbons--Warnick \cite{gibbons:2007} and Franchetti--Ross \cite{franchetti:2023aa} analysed the dynamics of well-separated monopoles, starting from a lagrangian for charged point particles.  Interestingly, they showed that the dynamics were governed by the LeBrun metrics, which had been studied much earlier in an entirely different context \cite{lebrun:1991a}.  Franchetti--Schroers \cite{franchetti:2017} calculated the $L ^2 $ metric for a circle-invariant 1-instanton on $S ^4  $, obtaining a finite metric on the moduli space of hyperbolic 1-monopoles which is however  invariant only under a rotational subgroup of the $H ^3 $ isometry group.   Finally, Sutcliffe \cite{sutcliffe:2022a} proposed a way to renormalise the divergent integral that defines the $L^2$ metric on the hyperbolic monopole moduli space.  His boundary metric, which had been considered earlier by Braam-Austin \cite{braam:1990}, is related to  the $L^2$ metric on the space of abelian connections on $S^2$.  This metric was calculated explicitly in the case of inversion-symmetric charge 2 monopoles \cite{sutcliffe:2022}.  The asymptotic geometry of the boundary metric is very different from the metric of Franchetti--Ross \cite{franchetti:2023aa,galvin:2024}, so it is difficult to reconcile Sutcliffe's metric with other approaches to monopole dynamics.  It is also not clear that the boundary metric admits any analogue of the hyperk\"ahler property.

\subsection{Main results}
This article introduces a new approach to the geometry of hyperbolic monopoles.  The starting point is a gauge-fixing condition introduced by Figueroa-O'Farrill--Gharamti \cite{figueroa-ofarrill:2014}.  On hyperbolic space with sectional curvature $-s^2$ it takes the form
\begin{equation}\label{GFintro}
\ast d^A \ast a + [\Phi,\phi]\pm {2i}{s}\phi=0.
\end{equation}
In this equation $a$ and $\phi$ are a Lie-algebra valued 1-form and 0-form describing a tangent vector to the moduli space.  In the limit $s=0$ it reduces to the gauge-fixing condition \eqref{coulombintro} used to define the hyperk\"ahler metric for euclidean monopoles.  This gauge-fixing condition was motivated by supersymmetry: it allowed the bosonic and fermionic zero-modes of the monopole to be related by a supersymmetry transformation.

Our main result is
\begin{theorem}\label{introthm1}
Let $(A,\Phi)$ be a $SU(2)$ hyperbolic monopole with half-integer mass.  Then every tangent vector to the framed moduli space can be represented by a pair $(a_\pm,\phi_\pm)$ that solves \eqref{GFintro}.  This representative is unique.  Moreover, its $L^2$ integral \eqref{metric integral} is finite and defines symmetric bilinear $\mathbb{C}$-valued pairings $g_\pm$ on the tangent space to the framed moduli space.
\end{theorem}
This result solves the problem of the divergent $L^2$ metric.  We also prove an analogous result for the unframed moduli space, see Corollary \ref{cor:unframed gauge fixing} below.  The assumption of half-integer mass means that $\|\Phi(x)\|$ approaches some constant $p\in\frac12\mathbb{N}$ as $x$ approaches the boundary of hyperbolic space.  The bilinear pairings $g_\pm$ in this theorem turn out to be complex conjugates of each other: $g_-=\overline{g_+}$.

The gauge-fixing condition \eqref{GFintro} is unusual in that it explicitly involves the complex number $i$. As a consequence, the solution $(a,\phi)$ take values in the Lie algebra of the complexification of the gauge group.  This is why the bilinear forms $g_\pm$ constructed in Theorem \ref{introthm1} are a priori $\C$-valued, rather than $\R$-valued.  We nevertheless show that $g_\pm$ are real in two special situations.

The first situation is where we restrict to the moduli space of monopoles of charge 1.  With our conventions this moduli space is diffeomorphic to $\R\times H^3$.  In section \ref{sec:calculating} we find some solutions of the gauge-fixing condition \eqref{GFintro}, valid for any charge $n$ and any mass parameter $p$, and hence calculate $g_\pm(X,X)$ explicitly for particular choices of $X$.  A special case of these calculations shows that, for $n=1$, $g_\pm$ is a constant real multiple of the product metric on $\R\times H^3$.

The second situation is where we restrict $g_\pm$ to the submanifold of the moduli space consisting of monopoles invariant under an inversion of hyperbolic space:
\begin{proposition}\label{introprop2}
$g_+ =g _- $ defines a real symmetric bilinear form on the moduli space of inversion-symmetric hyperbolic monopoles.
\end{proposition}
In particular, the hyperbolic analogue of the Atiyah--Hitchin metric is real.

Our second result concerns the structure of the tangent bundle of the moduli space:
\begin{theorem}\label{introthm3}
The complexified tangent bundle $T^\mathbb{C}M^f_{n,p}$ of the framed moduli space $M^f_{n,p}$ of hyperbolic monopoles with mass parameter $p\in\frac12\mathbb{N}$ and charge $n\in\mathbb{N}$ admits a natural decomposition
\begin{equation}
\label{decomposition intro}
T^\mathbb{C}M^f_{n,p} \cong E^\pm \otimes K^\pm,
\end{equation}
in which $K^\pm$ is a complex 2-dimensional vector space and $E^\pm\to M^f_{n,p}$ is a rank $2n$ complex vector bundle.
\end{theorem}
The bundles $E^\pm$ in \eqref{decomposition intro} are index bundles: their fibres are eigenspaces of a Dirac operator $D^A-\ad(\Phi)$ with eigenvalues $\pm\frac{i}{2}$.  The vector spaces $K^\pm$ are the spaces of Killing spinors on $H^3$.  There is a natural $L^2$ skew-symmetric pairing $\omega_{E^\pm}$ on the fibres of $E^\pm$, defined using the symplectic pairing on the spinor bundle of $H^3$, and also a natural symplectic pairing $\omega_{K^\pm}$ on $K^\pm$.

The decomposition \eqref{decomposition intro} induces a family of almost complex structures parametrised by lines in $K^\pm$.  For a given line $\langle\psi\rangle$ spanned by $\psi\in K^\pm$, the almost complex structure $J_{\langle\psi\rangle}$ is such that
\begin{equation}
T^{1,0}_{\langle\psi\rangle}M^f_{n,p} = E^+\otimes \langle\psi\rangle.
\end{equation}
Our third main result concerns the compatibility of these structures with $g^\pm$:
\begin{theorem}\label{introthm4}
The bilinear forms $g_\pm$, $\omega_{E^\pm}$, $\omega_{K^\pm}$ are related as follows:
\begin{equation}
g_\pm = \frac{1}{2}\omega_{E^\pm}\otimes \omega_{K^\pm}.
\end{equation}
The bilinear form $g_+$ and almost complex structure $J_{\langle\psi\rangle}$ are compatible in the sense that
\begin{equation}
g_+(X,X)=0\quad\forall X\in T^{1,0}_{\langle\psi\rangle}M^f_{n,p}.
\end{equation}
\end{theorem}

A decomposition similar to \eqref{decomposition intro} was obtained by Nash \cite{nash:2007} using less direct methods.  Rather than working directly with monopoles, Nash worked with spectral curves, which are related to monopoles via a twistor corresondence.  Nash constructed symplectic pairings on his analogues of $E^\pm$, while Bielawski--Schwachhofer \cite{bielawski:2013b,bielawski:2013} studied the families of complex structures derived from Nash's decomposition.
\begin{conjecture}
The decomposition \eqref{decomposition intro}, skew-symmetric pairings $\omega_{E^\pm}$ and almost complex structures $J_{\langle\psi\rangle}$ coincide with those of Nash and Bielawski--Schwachhofer.
\end{conjecture}
If true, this conjecture would imply that the $L^2$ metric $g_\pm$ is nondegenerate and that the almost complex structures $J_{\langle\psi\rangle}$ are integrable.  But the conjecture is not straightforward to prove, because the passage from monopoles to spectral curves is complicated (we comment on this more in Appendix \ref{appendix}).

Throughout this article we assume that the mass parameter $p\in\R_{>0}$ of the monopole is a half-integer.  This has two benefits.  First, it provides natural boundary conditions for \eqref{GFintro} by identifying hyperbolic monopoles with circle-invariant instantons on $S^4$, following \cite{atiyah:1987}.  Second, it allows for streamlined analytic proofs that exploit the compactness of $S^4$.  Relaxing the assumption $2p\in\mathbb{N}$ would result in lengthier proofs which we leave for future work.

\subsection{Outline}

This article is organised as follows.  Section \ref{sec:metric} defines the hyperbolic monopole moduli spaces and gives a more precise statement of Theorem \ref{introthm1} (see Theorem \ref{thm:framed gauge fixing}).  It also proves Proposition \ref{introprop2}.  The bulk of the proof of Theorem \ref{introthm1} is contained in section \ref{sec:existence}.  Section \ref{sec:Killing} proves Theorems \ref{introthm3} and \ref{introthm4} (see Theorem \ref{thm:decomposition} and Propositions \ref{prop:g J} and \ref{prop:g omega}).  It also completes the proof of Theorem \ref{introthm1} in Proposition \ref{prop:uniqueness}.  Section \ref{sec:calculating} contains some explicit calculations of $g_\pm$ in particular settings.  Finally, Appendix \ref{appendix} compares our new definition of the framed moduli space with others that appear in the literature, and compares the almost complex structures $J_{\langle\psi\rangle}$ with the work of Bielawski--Schwachhofer \cite{bielawski:2013b,bielawski:2013}.

\section{The monopole metric}
\label{sec:metric}

\subsection{Preliminaries}
A hyperbolic monopole is a pair $(A,\Phi)$ consisting of a connection $A$ and endomorphism $\Phi$ on a trivial vector bundle $V$ over hyperbolic space $H^3$ with sectional curvature $-1$.  These are required to satisfy the Bogomolny equation,
\begin{equation}\label{bog_eq}
d^A\Phi =\ast F^A,
\end{equation}
together with some boundary conditions to be specified.  In this paper attention will be restricted to gauge group $SU(2)$, so that $V$ has rank 2, $\Phi$ is an $\su(2)$-valued function, and $A$ is an $\mathfrak{su}(2)$-valued 1-form.  In this situation the boundary conditions require that:
\begin{itemize}
\item $V$, $\Phi$ and $A$ extend to the boundary $S^2_\infty$ of $H^3$;
\item $V|_{S^2_\infty}$ decomposes as a sum of line bundles $L^\pm$ with first Chern number $c _1(L^\pm) =\pm n$ for some $n\in\mathbb{N}$;
\item the bundles $L^\pm$ are eigenbundles of $\Phi_{S^2_\infty}$ with eigenvalues $\mp ip$ for some $p>0$;
\item the connection $A|_{S^2_\infty}$ is a sum of $U(1)$ connections on these line bundles.
\end{itemize}
The $SU(2)$-structure means that $L^-$ is naturally isomorphic to the dual of $L^+$, and that the connection on $L^-$ is the dual of the connection on $L^+$.  The integer $n$ is known as the charge of the monopole, while $p$ is known as the mass parameter.

Solutions of the Bogomolny equation \eqref{bog_eq} minimise the energy,
\begin{equation}\label{energy functional}
E = -\frac{1}{4}\int_{H^3}\tr(d^A\Phi\wedge\ast d^A\Phi + F^A\wedge \ast F^A).
\end{equation}
This fact follows from a simple integration by parts argument:
\begin{align}
E &= -\frac{1}{4}\int_{H^3}\tr(d^A\Phi-\ast F^A)\wedge\ast (d^A\Phi - \ast F^A) - \frac12\int_{H^3}\tr(d^A\Phi\wedge F^A)\\
&\geq 0 - \frac12\int_{S^2_\infty}\tr(\Phi F^A)
=2\pi np.\label{bog bound}
\end{align}
%For $n,p>0$ monopoles are irreducible, meaning that the bundle $\End_0(V)$ of traceless linear maps $V\to V$ has no nonzero parallel sections.  To see this, suppose that $\Psi$ is such a section.  Then $A=a\Psi+\frac{1}{2}\Psi^{-1}d\Psi$ and $F^A=da\,\Psi$ for some real 1-form $a$.  Over $S^2_\infty$, $\Phi$ must be a constant multiple of $\Psi$, and so $4\pi np = -\int_{S^2_\infty}\tr(\Phi F^A)$ is a multiple of $\int_{S^2_\infty}da = 0$.  So either $n=0$ or $p=0$.

We will see shortly that hyperbolic monopoles with $2p\in\mathbb{N}$ correspond to circle-invariant instantons on $S^4$ \cite{atiyah:1987}.  This is due in part to the conformal equivalence $S^1\times H^3$ and $S^4\setminus S^2$.  Let us construct this conformal equivalence explicitly, before explaining the relationship between instantons and monopoles.

We model $S^4$ as the set of unit vectors in $\R^5$.  By writing $y_1=r_1\cos\theta$, $y_2=r_1\sin\theta$ with $r_1>0$, the metric on $\mathbb{R}^5\setminus\mathbb{R}^3$, where $\mathbb{R}  ^3 \subset \mathbb{R}  ^5 $ is defined by $r_1=0$, is seen to be conformally equivalent to the metric on $S^1\times H^4$:
\begin{equation}\label{conformal S1H3}
\frac{1}{r_1^2}\sum_{\mu=1}^5 dy_\mu dy_\mu = d\theta^2 + \frac{1}{r_1^2}(dr_1^2+dy_3^2+dy_4^2+dy_5^2).
\end{equation}
The equation defining $S^4$, $r_1^2+y_3^2+y_4^2+y_5^2 =1$, determines an embedding $H^3\subset H^4$.  We thus obtain a conformal identification of $S^1\times H^3$ with $S^4\setminus S^2$. The action of the circle group on itself gives an action on $S^1\times H^3$ which in turn induces a rotation of $S^4$.  The fixed set $S^2\subset S^4$ of the circle action is naturally identified with the boundary of $H^3$.

We now explain how circle-invariant instantons give rise to hyperbolic monopoles.  Let $V\to S^4$ be a rank 2 bundle equipped with an $SU (2) $ lift of the following action of $\R$:
\begin{equation}\label{action on S^4}
t\cdot y = (y^1\cos(t)-y^2\sin(t),y^1\sin(t)+y^2\cos(t),y^3,y^4,y^5).
\end{equation}
Let $B$ be an instanton (that is, a connection with anti-self-dual curvature) on this bundle that is invariant under the action. We restrict $B$  to $S^4\setminus S^2 $ and pull it back to $\R\times H^3$ via the map
 \begin{equation}
 \mathbb{R}  \times H ^3 \rightarrow S ^1 \times H ^3 \rightarrow S ^4 \setminus S ^2 
 \end{equation} where the first map is given by $(\theta , x ) \mapsto (\mathrm{e} ^{ i\theta  },x )$ and the second one is the conformal equivalence between $S^1\times H^3$ and $S ^4 \setminus S ^2 $ discussed previously.
 This gives a connection on $\R\times H^3$ that is invariant under the action of translations.  After a gauge transformation, this connection  takes the form $\Phi\,d\theta+A$, where $(A,\Phi)$ are a connection and section over $H^3$, pulled back to $\R\times H^3$.  This connection has anti-self-dual curvature because the Hodge star is conformally invariant in middle dimension.  The anti-self-dual equation for the curvature of this connection is equivalent to the Bogomolny equation \eqref{bog_eq} for $(A,\Phi)$.  So $B$ determines a monopole provided that $(A,\Phi)$ satisfy the correct boundary conditions, which we verify in the next proposition.
\begin{proposition}\label{prop:S1 action}
Let $B$ be an irreducible anti-self-dual $SU(2)$ connection on a rank 2 bundle $V\to S^4$ that is invariant under an $SU(2)$ lift of the action \eqref{action on S^4}.  Then there exist $n,2p\in\mathbb{N}$ such that the following points hold.
\begin{enumerate}
\item The action of $2\pi\in\R$ on any fibre $V_y$ over $y\in S^4$ is multiplication by $(-1)^{2p}$. 
\item The action of $t\in\R$ on any fibre $V_y$ over a fixed point $y$ has eigenvalues $e^{\mp i p t}$.  The eigenbundles are line bundles $L^\pm$ over the fixed set $S ^2 $ such that $V|_{S^2}= L^+\oplus L^-$ and $c_1(L^\pm)=\pm n$. 
\item $B$ determines a hyperbolic monopole with charge $n$ and mass parameter $p$.
\item The second Chern number of $V$ is $2np$.
\end{enumerate}
\end{proposition}
\begin{proof}
Acting with $2\pi\in\R$ gives a section  of $\mathrm{Aut}(V)$, because  $2\pi$ fixes all points in $S^4$.  This section is parallel because the $\R$-action fixes the connection $B$.  Therefore its eigenvalues are constant.  If it had two distinct eigenvalues then the holonomy group of $B$ would be a subgroup of $U(1)$.  But this is not possible because $B$ is irreducible.  So the eigenvalues are equal.  Since the action is an $SU(2)$-action (meaning that it fixes a hermitian metric on $V$ and a non-vanishing section of $\Lambda^2V$), the determinant of the section of $\mathrm{Aut}(V)$ is 1.  So the eigenvalues are $\pm 1$ and the section is $\pm \mathrm{Id}$.

Now consider the restriction of $V$ to the fixed set $S^2$.  The action of $t\in\R$ gives an automorphism of $V|_{S^2}$.  Again, the eigenvalues are constant and their product is 1.  So they must be of the form $e^{\pm itp}$ for some $p\in\R$.  Since $t=2\pi$ acts as $\pm1$, the real number $p$ must satisfy $2p\in\mathbb{Z}$.  Without loss of generality we can assume $p\geq0$ and denote the eigenbundle for $e^{\pm itp}$ by $L^\mp$.  Let $n=c_1(L^+)\in\mathbb{Z}$.  Since $c_1(V)=0$, $c_1(L^-)=-n$.  Since the connection $B$ is invariant, its restriction to $S^2$ is a sum of connections on $L^\pm$.  The action of $t=2\pi$ on any fibre of $V|_{S^2}$ is given by multiplication with $(-1)^{2p}$, so by continuity, the same is true on any fibre of $V$.

Now we pull the bundle $V$ and connection $B$ back to $\R\times \overline{H}^3$, where $\overline{H}^3=H^3\cup S^2_\infty$.  The bundle $V|_{S^2}$ is trivial, so the pulled back bundle over $\R\times S^2_\infty$ is also trivial.  This trivialisation can be extended to a trivialisation of the bundle over $\R\times \overline{H}^3$.  The $\R$-action on this trivial bundle takes the form
\begin{equation}
t\cdot (\theta,x,v)= (\theta+t,x,h(t,\theta,x)v)\,\quad \forall (\theta,x,v)\in \R\times \overline{H}^3\times \C^2,
\end{equation}
for some $h(t,\theta,x)\in SU(2)$.  Our earlier calculation shows that $h(t,\theta,x)$ has eigenvalues $e^{\mp ipt}$ when $x\in S^2_\infty$.  We can make a gauge transformation $g$ so that the $\R$-action is trivial.  After doing so, the connection takes the form $\Phi\,d\theta+A$.  So we have that
\begin{align}
B&= g\Phi g^{-1}\,d\theta+gAg^{-1}-dg g^{-1}\label{B Phi A}\\
h(t,\theta,x) &= g(\theta+t,x)g(\theta,x)^{-1}.
\end{align}
When $x\in S^2_\infty$, the contraction of $B$ with $\partial/\partial\theta$ is zero, because $B$ was pulled back from a connection that extends across $S^2\subset S^4$.  So the equations together imply that
\begin{equation}
g\Phi g^{-1} = \frac{\partial g}{\partial\theta}g^{-1}=\frac{\partial}{\partial t}\bigg|_{t=0}h(t,\theta,x)
\end{equation}
for $x\in S^2_\infty$.  Therefore the eigenvalues of $\Phi$ over $S^2_\infty$ are the same as the eigenvalues of $\partial_th$ over $S^2\subset S^4$.  Since $h$ has eigenvalues $e^{\mp ipt}$ over $S^2$, the eigenvalues of $\Phi$ are $\mp ipt$, and the eigenspaces are identified with the eigenbundles $L^\pm$ of $h$.

%For $x\in S^2_\infty$ we know that $h(t,\theta,x)$ has eigenvalues $e^{\mp ipt}$, so
%\begin{equation}
%g(\theta,x)=\begin{pmatrix}e^{-ip\theta}&0\\0&e^{ip\theta}\end{pmatrix}g(0,x).
%\end{equation}
%with respect to the splitting $V|_{S^2_\infty}=L^+\oplus L^-$.  Inserting this into \eqref{B Phi A} and contracting with $\partial/\partial\theta$ gives
%\begin{equation}
%\Phi = \begin{pmatrix}-ip&0\\0& ip\end{pmatrix}.
%\end{equation}
%because $B(\partial/\partial\theta)=0$.  So $(A,\Phi)$ satisfies the boundary conditions for a monopole with mass parameter $p$.

By Chern-Weil theory and \eqref{bog bound},
\begin{multline}
c_2(V)=\frac{1}{8\pi^2}\int_{[0,2\pi]\times H^3}\tr(F^B\wedge F^B)=\frac{-1}{4\pi^2}\int_{[0,2\pi]\times H^3}d\theta\wedge\tr(d^A\Phi\wedge F^A)\\
=\frac{-1}{2\pi} \int_{H^3}\tr(d^A\Phi\wedge F^A)= 2np.
\end{multline}
\end{proof}
Now we describe the reverse process.  Let $(A,\Phi)$ be a hyperbolic monopole; we will attempt to construct a circle-invariant instanton.  We choose a gauge transformation $g:\R\times H^3\to SU(2)$ and define a connection $B$ on $\R\times H^3$ via \eqref{B Phi A}.  We want this connection to be periodic in $\theta$; for this it is sufficient to choose $g$ such that
\begin{equation}
g(\theta+2\pi,x)=\pm g(\theta,x).
\end{equation}
With this choice, $B$ descends to an anti-self-dual connection on $S^1\times H^3\simeq S^4\setminus S^2$.  We now attempt to extend this connection to $S^4$.  The holonomy of the connection along $S^1\times\{x\}$ is 
\begin{equation}
g(2\pi,x)^{-1}\exp(-2\pi\Phi(x))g(0,x).
\end{equation}
As $x$ approaches $S^2_\infty$ the eigenvalues of this matrix tend to $\pm e^{2\pi ip}$ and $\pm e^{-2\pi ip}$.  A theorem of Sibner--Sibner shows that the connection extends to $S^4$ if and only if the eigenvalues approach 1 \cite{sibner:1988,sibner:1992}.  In the ``$+$'' case this means that $p$ is an integer, and in the ``$-$'' case it means that $p\in (\frac12\mathbb{Z})\setminus\mathbb{Z}$.  Thus:
\begin{proposition}
A hyperbolic monopole extends to a circle-invariant instanton on $S^4$ if and only if $2p\in\mathbb{N}$.
\end{proposition}

\subsection{The framed and unframed moduli spaces}
The Bogomolny equation is invariant under gauge transformations, and it is natural to study the moduli space of solutions modulo gauge transformations.  In fact, there are two natural moduli spaces to consider, the framed and unframed moduli spaces.  We denote by $C_{n,p}$ the space of solutions of the Bogomolny equations and boundary conditions with charge $n$ and mass parameter $p$ that extend smoothly to the bundle $V\to S^4$.  We define the group of \emph{unframed gauge transformations} to be the group $G^u$ of functions $g:H^3\to SU(2)$ that extend to smooth circle-invariant sections of the bundle $\Aut(V)\to S^4$.  Note that circle-invariance demands that on the fixed $S^2\subset S^4$, $g$ maps $L^+$ to $L^+$ and $L^-$ to $L^-$.  Two monopoles are called \emph{unframed gauge equivalent} if they are related by an unframed gauge transformation.  The moduli space $M^u_{n,p}$ of \emph{unframed} monopoles is the quotient of $C_{n,p}/G^u$.  This is known to be a smooth manifold of dimension $4n-1$ \cite{braam:1990}.

To define the framed moduli space, we need to recall the definition of the Chern--Simons number.  Let $A^0,A^1$ be two connections on $H^3$ and let $\tilde{A}=(1-t)A^0+tA^1$ be a connection on $[0,1]\times H^3$.  The Chern-Simons number of $A^0,A^1$ is defined to be
\begin{equation}\label{chern-simons definition}
CS[A^0,A^1]=\int_{[ 0,1] \times H^3}\tr(F^{\tilde A}\wedge F^{\tilde A}).
\end{equation}
By explicitly evaluating the integral in $t$, one can obtain a more explicit formula for the Chern-Simons number:
\begin{multline}\label{chern-simons explicit}
CS[A^0,A^1]=\int_{H^3}\tr\left(A^1\wedge dA^1 + \frac23 A^1\wedge A^1\wedge A^1\right)\\ - \int_{H^3}\tr\left(A^0\wedge dA^0 + \frac23 A^0\wedge A^0\wedge A^0\right)+\int_{S^2_\infty}\tr(A^0\wedge A^1).
\end{multline}
\begin{lemma}\label{lemma:chern simons additivity}
Let $A^0,A^1,A^2$ be connections satisfying the boundary conditions of a monopole, and let $g:H^3\to SU(2)$ be an unframed gauge transformation.  The Chern--Simons number is gauge-invariant, in the sense that
\begin{equation}
CS[A^0,A^1]=CS[g^{-1}A^0g+g^{-1}dg,g^{-1}A^1g+g^{-1}dg].
\end{equation}
It is skew-symmetric, meaning that
\begin{equation}\label{chern-simons skew symmetry}
CS[A^1,A^0]=-CS[A^0,A^1].
\end{equation}
Moreover, if $A^0,A^1,A^2$ are unframed gauge equivalent then
\begin{equation}\label{chern-simons additivity}
CS[A^0,A^2]=CS[A^0,A^1]+CS[A^1,A^2].
\end{equation}
\end{lemma}
\begin{proof}
Gauge invariance is immediate from the definition \eqref{chern-simons definition}, as $F^{\tilde{A}}\mapsto g^{-1}F^{\tilde{A}}g$ under gauge transformations.  Skew-symmetry is immediate from \eqref{chern-simons explicit}, cyclicity of the trace, and skew-symmetry of the wedge product.  For the final result \eqref{chern-simons additivity} we first compute
\begin{equation} 
\begin{split} \label{CS additivity computation}
CS[A^0,A^1]+&CS[A^1,A^2]-CS[A^0,A^2]\\
&= \int_{S^2_\infty}\tr(A^0\wedge A^1+A^1\wedge A^2-A^0\wedge A^2)\\
&= \int_{S^2_\infty}\tr((A^1-A^0)\wedge(A^2-A^0)-A^0\wedge A^0).
\end{split} 
\end{equation}
The second of these two terms vanishes by cyclicity of the trace and skew-symmetry of the wedge product.  To analyse the first we first introduce the section $T$ of $\End(V)$ over $S^2_\infty$ such that $Tv=\pm v$ for $v\in L^\pm$.  Since the connections $A^i$ respect the splitting $V=L^+\oplus L^-$, they make $T$ parallel.  It follows that
\begin{equation}
A^j=\frac12 T^{-1}dT + T\alpha^j
\end{equation}
for 1-forms $\alpha^j$ on $L^+$.  By hypothesis,  the connections $A^1,A^2$ differ from $A^0$ by unframed gauge transformations $g^j$ with $j=1,2$.  Over $S^2_\infty$, these gauge transformations take the form $g^j=\exp(i\lambda^jT)$ for functions $\lambda^j:S^2_\infty\to\R$.  So
\begin{equation}
A^j-A^0=iTd\lambda^j
\end{equation}
over $S^2_\infty$.  It follows from \eqref{CS additivity computation} that
\begin{multline}
CS[A^0,A^1]+CS[A^1,A^2]-CS[A^0,A^2]=-2\int_{S^2_\infty}d\lambda^1\wedge d\lambda^2\\
=-2\int_{S^2_\infty}d(\lambda^1\wedge d\lambda^2)=0.
\end{multline}
%
% $A^\alpha=g_\alpha^{-1} A^0g_\alpha+g_{\alpha}^{-1}dg_\alpha$ for unframed gauge transformations $g_\alpha$ with $\alpha=1,2$.  Then
%\begin{equation}
%A^\alpha-A^0=g_\alpha^{-1}d^{A^0}g_\alpha.
%\end{equation}
%Since $g_\alpha$ are circle-invariant, they respect the splitting $E=L^+\oplus L^-$ over $S^2_\infty$ and we can find functions $\lambda_\alpha:S^2_\infty\to\R$ such that
%\begin{equation}
%g_\alpha = \exp\begin{pmatrix}i\lambda_\alpha & 0 \\ 0 & -i\lambda_\alpha\end{pmatrix}
%\end{equation}
%with respect to this splitting.  The boundary conditions require that $A^0$ reduces to $U(1)$ connections on $L^\pm$ over $S^2_\infty$.  Since $U(1)$ is abelian,
%\begin{equation}
%g_\alpha^{-1}d^{A^0}g_\alpha=\begin{pmatrix}id\lambda_\alpha&0\\0&-id\lambda_\alpha\end{pmatrix}.
%\end{equation}
%So
%\begin{multline}
%CS[A^0,A^1]+CS[A^1,A^2]-CS[A^0,A^2]=-2\int_{S^2_\infty}d\lambda^1\wedge d\lambda^2\\
%=-2\int_{S^2_\infty}d(\lambda^1\wedge d\lambda^2)=0.
%\end{multline}
\end{proof}

We say that two connections $(A^1,\Phi^1),(A^2,\Phi^2)$ are \emph{framed gauge equivalent} if they are unframed gauge equivalent and $CS[A^1,A^2]=0$.  It follows from the previous lemma that framed gauge equivalence is an equivalence relation.  We define the moduli space $M^f_{n,p}$ of \emph{framed monopoles} to be the quotient of $C_{n,p}$ by framed gauge equivalence.

There is a natural map
\begin{equation}\label{projection}
M^f_{n,p}\to M^u_{n,p}
\end{equation}
that sends a framed monopole to its unframed gauge equivalence class.  There is also natural action of $\R$ on the framed moduli space that is generated by infinitesimal gauge transformation with $\Phi$.  Concretely, the action of $s\in\R$ is
\begin{equation}\label{action}
(A,\Phi)\mapsto (A^s,\Phi):=(g_s^{-1}dg_s+g_s^{-1}Ag_s,\Phi);\quad g_s:=\exp(s\Phi).
\end{equation}
This clearly preserves the fibres of \eqref{projection}, because $g_s$ is an unframed gauge transformation.  The following lemma shows that this action is free.
\begin{lemma}\label{lemma:changing CS}
Let $(A,\Phi)$ be a monopole, let $g_s=\exp(s\Phi)$ and let $A^s=g_s^{-1}dg_s+g_s^{-1}Ag_s$.  Then
\begin{equation}
\label{varAundergt} 
\frac{d A ^s }{ds}= d^{A  ^s }\Phi 
\end{equation} 
and
\begin{equation}
\label{chernpuregaugetrasf} 
CS[A,A^s]=-8\pi n p s.
\end{equation}
In particular, $A^s$ is framed gauge equivalent to $A$ if and only if $s=0$.
\end{lemma}
\begin{proof}
We calculate
\begin{equation}
\frac{d}{ds}g_s=g_s\Phi,\quad\frac{d}{ds}g_s^{-1}=-\Phi g_s^{-1}\quad\text{and}\quad\frac{d}{ds}dg_s=g_sd\Phi+dg_s\Phi.
\end{equation}
So
\begin{equation}
\frac{d}{ds}A^s = g_s^{-1}(g_sd\Phi+dg_s\Phi+Ag_s\Phi)-\Phi g_s^{-1}(dg_s+A_s\Phi)
=d^{A^s}\Phi.
\end{equation}
We prove (\ref{chernpuregaugetrasf})  by direct calculation from the explicit formula \eqref{chern-simons explicit} (one can also prove this directly from the definition \eqref{chern-simons definition}),
\begin{equation} 
\begin{split} 
\frac{d}{ds}CS[A,A^s]
&=\int_{H^3}\tr(d^{A^s}\Phi\wedge dA^s + A^s\wedge d(d^{A^s}\Phi)+2A^s \wedge A^s \wedge d^{A^s}\Phi)
\\
&\quad 
+ \int_{S^2_\infty}\tr(A\wedge d^{A^s}\Phi) \\
&=\int_{H^3}\big[ 2\tr(d^{A^s}\Phi\wedge F^{A^s})-d\tr(A^s\wedge d^{A^s}\Phi)\big] + \int_{S^2_\infty}\tr(A\wedge d^{A^s}\Phi)\\ &
=\int_{H^3}2\tr(d^{A^s}\Phi\wedge F^{A^s}).
\end{split}
\end{equation}
since the boundary conditions imply that $d^{A^s}\Phi=0$ on $S^2_\infty$.  Therefore, by \eqref{bog bound},
\begin{equation}
\frac{d}{ds}CS[A,A^s]=-8\pi np
\end{equation}
The result follows.
\end{proof}

\begin{proposition}\label{prop:principal}
The moduli space of framed monopoles is naturally an $\R$-principal bundle over the moduli space of unframed monopoles.
\end{proposition}
\begin{proof}
We have already exhibited a projection \eqref{projection} from $M^f_{n,p}$ to $M^u_{n,p}$ and an action \eqref{action} of $\R$ on the fibres.  We showed in the previous lemma that the action is free.  It remains to show that the action is transitive.

Let $(A^0,\Phi^0)$ and $(A^1,\Phi^1)$ be two monopoles that are unframed gauge equivalent.  We seek an $s$ such that the action \eqref{action} of $s$ on $(A^0,\Phi^0)$ gives a monopole which is framed gauge equivalent to $(A^1,\Phi^1)$.  Let $(A^s,\Phi^s)=(g_s^{-1}A^0g_s+g_s^{-1}dg_s,\Phi^0)$ denote the image of $(A^0,\Phi^0)$ under this action.  Using lemma \eqref{lemma:changing CS}, we can choose $s$ such that
\begin{equation}
CS[A^0,A^s]=CS[A^0,A^1].
\end{equation}
Since $A^0,A^1,A^s$ are unframed gauge equivalent, we can apply the identity \eqref{chern-simons additivity} to show that
\begin{equation}
CS[A^s,A^1]=CS[A^0,A^1]-CS[A^0,A^s]=0.
\end{equation}
So $(A^s,\Phi^s)$ and $(A^1,\Phi^1)$ are framed gauge equivalent.
\end{proof}
\begin{corollary}\label{cor:dimension Mf}
$\dim(M^f_{n,p})=4n$.
\end{corollary}
\begin{proof}
Proposition \ref{prop:principal} shows that $\dim(M^f_{n,p})=\dim(M^u_{n,p})+1$, and Braam \cite{braam:1989} has shown that shown that $\dim(M^u_{n,p})=4n-1$.
\end{proof}

Our definition of the framed moduli space is different from the definition given in Atiyah--Hitchin \cite{atiyah:2014} for euclidean monopoles.  We compare the two definitions in appendix \ref{appendix}.  Atiyah--Hitchin's framed moduli space involves a choice of point in $S^2_\infty$, so it does not obviously admit an action of the isometry group of $H^3$.  For euclidean monopoles, they showed that their framed moduli space does admit an action of the euclidean isometry group \cite{atiyah:2014}, but their argument does not adapt easily to the hyperbolic setting.  In contrast, our definition leads to a framed moduli space that clearly admits an action of the isometry group of $H^3$.

Our definition of the framed moduli space is motivated by monopole dynamics.  One can extend the energy functional \eqref{energy functional} to define a lagrangian gauge theory on $\R\times H^3$.  Dynamical configurations consist of a gauge field $A$ and scalar field $\Phi$ on $\R\times H^3$.  Then, for fixed $t_0,t_1\in\R$, the quantity
\begin{equation}
\Theta(t_0,t_1) = \int_{[t_0,t_1]\times H^3}\tr(F^A\wedge F^A)
\end{equation}
is clearly gauge-invariant and equal to 0 for any static configuration (i.e.\ any configuration that is pulled back from $H^3$).  By working in a gauge where $A_t=0$, slowly moving dynamical configurations can be approximated by paths in the monopole moduli space, parametrised by $t\in\R$. Consider a path of the form $(A^t,\Phi^t)=(g_t^{-1}Ag_t+g_t^{-1}dg_t,g_t^{-1}\Phi g_t)$, where $(A,\Phi)$ is a fixed monopole on $H^3$ and $g$ is a $t$-dependent gauge transformation.  It can be shown that $\Theta(s,t)=CS[A^s,A^t]$.  In particular, Lemma \ref{lemma:changing CS} shows that, for the choice $g_t=\exp(t\Phi)$,  $\Theta$ is non-zero.  So this is not a static monopole.  The corresponding path $\R\to M^f_{n,p}$ is non-constant, but the corresponding path $\R\to M^u_{n,p}$ is constant.  So the unframed moduli space is unable to distinguish this dynamical monopole from a static monopole, but the framed moduli space is able to do so.

Paths in the moduli space generated by gauge transformations with $\Phi$ are known as dyons and carry an electric charge, which is defined to be the abelian electric flux through $S^2_\infty$ and computed by the integral $\frac{1}{2p}\int_{S^2_\infty}\tr(\Phi\ast\frac{d A}{dt})$.  For $(A^t,\Phi^t)=(g_t^{-1}Ag_t+g_t^{-1}dg_t,g_t^{-1}\Phi g_t)$, by (\ref{varAundergt})   $\frac{d A ^t }{dt}= d^{ A ^t }\Phi$. Denoting by $F ^t = g _t ^{-1} F g _t $ the curvature of $A _t $, the electric flux is  $\frac{1}{2p}\int_{S^2_\infty}\tr(\Phi ^t  F^t  ) =\frac{1}{2p}\int_{S^2_\infty}\tr(\Phi   F  ) =- 2 \pi  n$ and by \eqref{chernpuregaugetrasf} the rate of change of the Chern-Simons number is $-8 \pi n p $.  Hence the electric charge is proportional to the rate of change of the Chern-Simons number.

For much of this article we will work with tangent vectors to the framed (and unframed) moduli spaces.  The tangent spaces can be described as follows.  First, a tangent vector to the configuration space $C_{n,p}$ at a point $(A,\Phi)$ consists of an $\su(2)$-valued 1-form $a$ and 0-form $\phi$ solving the linearised Bogomolny equation,
\begin{equation}\label{linearised bog_eq}
d^A\phi+[a,\Phi]=\ast d^A a,
\end{equation}
such that $b=a+\phi\,d\theta$ extends to a smooth circle-invariant section of $\Lambda^1(S^4)\otimes\End(V)$.  Given any $\su(2)$-valued function $\chi$ that extends to a circle-invariant section of $\End(V)$ over $S^4$, $(a,\phi):=(d^A\chi,[\Phi,\chi])$ is a solution of \eqref{linearised bog_eq} representing an infinitesimal unframed gauge transformation.  Thus we define the space $\g^u_{A,\Phi}$ of infinitesimal unframed gauge transformations to be the set of all pairs $(d^A\chi,[\Phi,\chi])$ such that $\chi:H^3\to\su(2)$ extends smoothly to a circle-invariant section of $\End(V)\to S^4$.\label{gu definition}

In order to define infinitesimal framed gauge transformations we look at the linearisation of the framing condition. Let  $g_s$ be a family of framed gauge transformations such that  $g _0 =1 $, $g_s^{-1}\partial_s g_s|_{ s =0 }=\chi$ and set $A ^s = g_s^{-1}Ag_s+g_s^{-1}dg_s$. Since $g  _s $ is framed,  $CS[A,A ^s ]=0 $. A computation similar to the one in the proof of Lemma \ref{lemma:changing CS} then shows that 
\begin{equation}
0=\left.\frac{d}{ds}\right|_{s=0}CS[A,A^s] = \int _{ H ^3 } 2\tr (F^A  \wedge d ^A \chi ).
\end{equation}
Hence we define the space $\g^f_{A,\Phi}$ of infinitesimal framed gauge transformations to be  
\begin{equation} 
\label{framedinfinitesimalgt} 
\g^f_{A,\Phi}=\left\{ (d ^A \chi , [ \Phi , \chi ]) \in \g^u _{  A , \Phi }  :\int_{H^3}\tr( F^A \wedge d^A\chi )=0\right \}.
\end{equation} 
Both $\g^u_{A,\Phi}$ and $\g^f_{A,\Phi}$ are subspaces of the space $T_{A,\Phi}C_{n,p}$ of solutions $(a,\phi)$ of \eqref{linearised bog_eq} that extend smoothly to $S^4$ in a circle-invariant way.  The tangent spaces to the framed and unframed moduli spaces are the quotient spaces
\begin{equation}\label{tangent space definition}
T_{A,\Phi}M^f_{n,p}=T_{A,\Phi}C_{n,p}/\g^f_{A,\Phi},\quad T_{A,\Phi}M^u_{n,p}=T_{A,\Phi}C_{n,p}/\g^u_{A,\Phi}.
\end{equation}

\subsection{The moduli space metric}
An $L^2$ metric on the moduli space is a metric obtained by integrating the pointwise inner product of a tangent vector $X=(a,\phi)$ with itself:
\begin{equation}\label{metric integral}
g(X,X)=-\frac{1}{2}\int_{H^3}\tr(a\wedge\ast a+\phi\ast \phi).
\end{equation}
In order to make this well-defined, we must choose a particular representative $(a,\phi)$ of a gauge equivalence class; in other words, we must fix a gauge.  Unfortunately, any real representative of a non-zero tangent vector results in a divergent integral.  This is because unframed hyperbolic monopoles have the holographic property of being determined by the abelian connection induced at infinity \cite{braam:1990}.  Any representative $a$ of a non-zero tangent vector must be non-zero when restricted to the line bundles $L^\pm$ over $S^2_\infty$.  It follows that $a'=-\tr(a\Phi)/\sqrt{-\tr(\Phi^2)}$ gives a non-vanishing 1-form on any geodesic 2-sphere of sufficiently large radius.  The metric $dr^2+\sinh^2 (r)  g_{S^2}$ of $H^3$ is such that the 1-form $a'$ on $S^2$ satisfies $\ast a' = \ast_{S^2}a'\wedge dr$.  Hence there exist positive constants $R,C$ such that
\begin{equation}
\label{divergence of the usual metric} 
g(X,X)\geq \frac{1}{2}\int_{r\geq R} a'\wedge\ast_{S^2} a'\wedge dr \geq \int_{R}^\infty C\,dr = \infty.
\end{equation}
Thus the integral determining the metric is divergent. 

One way to avoid this divergence is to choose a non-real representative of a real tangent vector.   Using an $\mathfrak{sl}(2,\C)$-valued gauge transformation, we can choose a representative $(a,\phi)$ that is not antihermitian, but is such that $(a+a^\dagger,\phi+\phi^\dagger)$ is pure gauge.   The advantage of doing so is that it opens up the possibility of $a|_{S^2_\infty}$ being a holomorphic (or antiholomorphic) 1-form.  A 1-form $a_\infty$ that is of type $(1,0)$ or $(0,1)$ with respect to the complex structure of $S^2_\infty$ satisfies $a_\infty\wedge \ast_{S^2}a_\infty=0$.  Thus the divergence of the $L^2$ metric could be avoided if the restriction of $a$ to $S^2_\infty$ were holomorphic or antiholomorphic.  Note that this argument only applies to the complex-bilinear form $g$ defined in \eqref{metric integral}; it does not apply to the hermitian metric $-\int\tr(a\wedge\ast\bar{a}+\phi\ast\bar{\phi})$ because $a_\infty\wedge\ast_{S^2}\bar{a}_\infty=0$ only when $a_\infty=0$.

A non-real gauge-fixing condition has been proposed by Gharamti--Figueroa-O'Farrill \cite{figueroa-ofarrill:2014}.  It takes the form
\begin{equation}\label{gauge fixing}
\ast d^A \ast a_\pm + [\Phi,\phi_\pm]\pm 2i\phi_\pm=0.
\end{equation}
It differs from the more familiar harmonic gauge-fixing condition used in \cite{braam:1990,sutcliffe:2022} by the term $\pm 2i\phi_\pm $.  Equation \eqref{gauge fixing} is not one but two gauge-fixing conditions corresponding to the choice of sign.  The two representatives $(a_\pm,\phi_\pm)$ of a real tangent vector to the moduli space will differ by a gauge transformation.  We are free to choose $(a_+,\phi_+)=(-a_-^\dagger,-\phi_-^\dagger)$, as this choice clearly solves the ``+'' gauge fixing condition if $(a _- , \phi _- )$ solves the ``$-$'' condition, and, if $(a_-,\phi_-)$ represents a real tangent vector, it differs from $(-a_+^\dagger,-\phi_+^\dagger)$ by a gauge transformation --- compare with  (\ref{a0phi0}) below.

The gauge-fixing condition \eqref{gauge fixing} was not motivated by the divergence of moduli space metrics but instead arose from a study of supersymmetry transformations.  Indeed, comments in the introduction of \cite{figueroa-ofarrill:2014} suggest that the authors had not suspected that this gauge-fixing constraint would cure the metric divergence.  Nevertheless, we will show below that this gauge-fixing constraint leads to a satisfying regularisation of the divergent metric.  We will show in the next section that:

\begin{theorem}\label{thm:gauge fixing}
Let $(A,\Phi)$ be a hyperbolic monopole with half-integer mass and let $(a,\phi)$ be a solution of the linearised Bogomolny equation \eqref{linearised bog_eq} that represents a tangent vector to the moduli space.  Then there exists an infinitesimal unframed $\mathfrak{sl}(2,\C)$ gauge transformation $\chi$ such that $(a_-,\phi_-)=(a+ d^A\chi,\phi+[\Phi,\chi])$ satisfies the gauge-fixing condition \eqref{gauge fixing}.  This $\chi$ has the following properties:
\begin{enumerate}
\item Over $S^2_\infty$, $a_-$ is a section of $\Lambda^{1,0}$.
\item The $L^2$ norm \eqref{metric integral} of $(a_-,\phi_-)$ is finite.
\end{enumerate}
\end{theorem}
In order to apply this theorem to the framed  moduli space, we need the gauge transformation $\chi$ to be framed.  But it is easy to modify $\chi$ so that it is framed.  If $(a_-,\phi_-)$ solves the gauge-fixing condition \eqref{gauge fixing} then so does $(a_--c\,d^A\Phi,\phi_-)$ for any constant $c$.  So replacing $\chi$ in Theorem \ref{thm:gauge fixing} with $\tilde{\chi}=\chi-c\Phi$ and setting
\begin{equation}
c = \int_{H^3}\tr(d^A\chi\wedge F^A )/\int_{H^3}\tr(d^A\Phi\wedge F^A )
= - \frac{1}{4 \pi np} \int_{H^3}\tr(d^A\chi\wedge F^A )
\end{equation}
produces a framed gauge transformation that still solves \eqref{gauge fixing}.  Hence from Theorem \ref{thm:gauge fixing} we also obtain the existence part of the following:
\begin{theorem}\label{thm:framed gauge fixing}
Let $(A,\Phi)$ and $(a,\phi)$ be as in Theorem \ref{thm:gauge fixing}.  Then there exists a unique infinitesimal framed gauge transformation $ \chi$ such that $(a_-,\phi_-)=(a+ d^A \chi,\phi+[\Phi, \chi])$ satisfies the gauge-fixing condition \eqref{gauge fixing}.  This $ \chi$ has the two properties listed in Theorem \ref{thm:gauge fixing}.
\end{theorem}
We will show that $(a_--c\,d^A\Phi,\phi_-)$ has the two properties listed in Theorem \ref{thm:gauge fixing} at the end of section \ref{sec:existence}.  The uniqueness part of this Theorem will be proved in Proposition \ref{prop:uniqueness} in section \ref{sec:Killing}.  Turning to the unframed moduli space, we have that:
\begin{corollary}\label{cor:unframed gauge fixing}
Let $(A,\Phi)$ and $(a,\phi)$ be as in Theorem \ref{thm:gauge fixing}.  Then there exists a unique infinitesimal unframed gauge transformation $\chi$ such that $(a_-,\phi_-)=(a+ d^A\chi,\phi+[\Phi,\chi])$ satisfies the gauge-fixing condition \eqref{gauge fixing} and $\int_{H^3}\tr(a_-\wedge F^A )=0$.  This $\chi$ has the two properties listed in Theorem \ref{thm:gauge fixing}.
\end{corollary}
\begin{proof}
We may write the unframed gauge transformation in the form $\chi=\chi_0-cd^A\Phi$, in which $\chi_0$ is a framed gauge transformation and $c\in\C$.  The requirement $\int_{H^3}\tr(a_-\wedge F^A )=0$ is met if and only if $c=\int_{H^3}\tr(a\wedge F^A )/\int_{H^3}\tr(d^A\Phi\wedge F^A )$.  Then Theorem \ref{thm:framed gauge fixing} ensures existence and uniqueness of a suitable framed gauge transformation $\chi_0$.
\end{proof}

Theorem \ref{thm:framed gauge fixing} ensures that calculating the bilinear form (\ref{metric integral}) for $(a, \phi )$ solutions of   the gauge fixing condition \eqref{gauge fixing} yields two symmetric bilinear forms $g_\pm$ on the tangent bundle of the framed monopole moduli space. They are given by
%Theorem \ref{thm:framed gauge fixing} yields two symmetric bilinear forms $g_\pm$ on the tangent bundle of the framed monopole moduli space.  They are defined by
\begin{equation}\label{g+-}
g_\pm(X,X)=-\frac{1}{2}\int_{H^3}\tr(a_\pm\wedge\ast a_\pm+\phi_\pm\ast \phi_\pm),
\end{equation}
for $(a_-,\phi_-)$  the solution of the gauge fixing condition \eqref{gauge fixing} determined in Theorem \ref{thm:framed gauge fixing}, and $(a_+,\phi_+)=(-a_-^\dagger,-\phi_-^\dagger)$.  By construction, the bilinear forms $g_\pm$ are also related by conjugation: in other words,
\begin{equation}\label{metric conjugation}
g_-(X,X)=\overline{g_+(X,X)}
\end{equation}
for any real tangent vector $X$.  We recall that we call a tangent vector $X =(a, \phi )$  real if  $a + a ^\dagger $, $\phi + \phi ^\dagger $ are pure gauge.

We have been careful to avoid calling $g_\pm$ metrics because we do not know whether they are always real-valued, and if so, whether they are positive definite.  They are real if and only if $g_+=g_-$, so we can investigate reality by calculating their difference.  To do this, it is convenient to write
\begin{equation}\label{a0phi0}
(a_\pm,\phi_\pm)=(a_0\pm id^A\lambda,\phi_0\pm i[\Phi,\lambda]),
\end{equation}
in which $a_0,\phi,\lambda$ are $\su(2)$-valued.  The gauge-fixing condition \eqref{gauge fixing} is then equivalent to
\begin{align}
\ast d^A\ast a_0+[\Phi,\phi_0]-2[\Phi,\lambda]&=0,\label{gauge fixing real}\\
\ast d^A\ast d^A\lambda+[\Phi,[\Phi,\lambda]]+2\phi_0&=0.
\end{align}
A short calculation shows that
\begin{align}
g_+(X,X)-g_-(X,X) &= -2i\int_{H^3}\tr([\Phi,\lambda]\ast\phi_0 + d^A\lambda\wedge\ast a_0)\\
&=2i\int_{H^3}\tr (\lambda([\Phi,\ast\phi_0] + d^A\ast a_0))-d\tr(\lambda\ast a_0)\\
&=4i\int_{H^3}\tr \ast(\lambda[\Phi,\lambda])-2i\int_{S^2_\infty}\tr(\lambda\ast a_0)\\
&=-2i\int_{S^2_\infty}\tr(\lambda\ast a_0)\label{g imaginary part}.
\end{align}
Thus the imaginary part of $g_+$, which is proportional to $g_+-g_-$, reduces to an integral over the boundary of $H^3$.  However, we have not been able to show in general that this boundary integral vanishes.

More can be said in the case of inversion-symmetric monopoles.  Recall that inversion about a point in $H^3$ is the unique orientation-reversing isometry $I:H^3\to H^3$ whose fixed set is that point and which squares to the identity.  A monopole is called inversion-symmetric if
\begin{equation}
(A,\Phi)=(I^\ast A,-I^\ast\Phi)
\end{equation}
up to gauge equivalence (the minus sign in this equation ensures that the right-hand-side solves \eqref{bog_eq}).  Now suppose that $(a_\pm,\phi_\pm)$ solve the gauge-fixing condition \eqref{gauge fixing} and represent a chosen real tangent vector $X$ to the moduli space of inversion-symmetric monopoles.  Then $(a_\pm^I,\phi_\pm^I):=(I^\ast a_\pm,-I^\ast\phi_\pm)$ differ from $(a_\pm,\phi_\pm)$ by gauge transformations.  It is straightforward to check that
\begin{equation}
\ast d^A \ast a_\pm^I + [\Phi,\phi_\pm^I]\mp 2i\phi_\pm^I=0.
\end{equation}
By the uniqueness part of Theorem \ref{thm:framed gauge fixing}, this means that $(a_\pm^I,\phi_\pm^I)=(a_\mp,\phi_\mp)$, in other words,
\begin{equation}
I^\ast a_\pm=a_\mp\quad\text{and}\quad I^\ast\phi_\pm=-\phi_\mp.
\end{equation}
Since the $L^2$ inner product \eqref{g+-} is invariant under inversions, this means that
\begin{equation}
g_-(X,X)=g_+(X,X)
\end{equation}
for all real tangent vectors $X$ to the moduli space of inversion-symmetric monopoles.  Combined with the relation \eqref{metric conjugation}, this shows that
\begin{proposition}
$g_+ =g _- $ defines a real symmetric bilinear form on the moduli space of inversion-symmetric hyperbolic monopoles.
\end{proposition}

A similar argument shows that the restriction of $g_+$ to the normal bundle of the moduli space of inversion-symmetric monopoles is also real.  For normal vectors one has that $(a_\pm^I,\phi_\pm^I)$ is gauge equivalent to $(-a_\pm,-\phi_\pm)$ rather than $(a,\phi)$, and hence that $I^\ast a_\pm=-a_\mp$ and $I^\ast\phi_\pm=+\phi_\mp$, leading to the conclusion that $g_+(X,X)=g_-(X,X)$ for normal tangent vectors $X$.  In the particular case of charge 1, it is known that every monopole is inversion-symmetric about some point and that the moduli-space of inversion-symmetric monopoles about a given point is 0-dimensional.  Therefore every tangent vector is normal to some inversion-symmetric moduli space and thus $g_\pm$ is a real symmetric bilinear form on moduli space of 1-monopoles.  We will calculate this metric explicitly in section \ref{sec:calculating} and confirm that it is positive-definite.

Turning our attention to the unframed moduli spaces, Corollary \ref{cor:unframed gauge fixing} allows us to define a pair of complex-conjugate symmetric bilinear forms via the integral \eqref{g+-}.  Equivalently, these bilinear forms could be obtained from those on the framed moduli spaces by the quotient construction.  Again, we do not know whether these are real, but they are real when restricted to inversion-symmetric monopoles, and they are real and positive definite on the unframed moduli space of 1-monopoles as we show in Section \ref{sec:calculating}.

\begin{remark}
The reader may wonder whether the constant $\pm 2i$ in the gauge fixing condition \eqref{gauge fixing real} could be replaced by another constant.  The choice $\pm 2i$ is special because it is linked to supersymmetry and leads to special geometry, which will be described in Section \ref{sec:Killing}.  Other choices close to $\pm 2i$ might lead to convergent metrics, but those metrics would not be expected to have any special geometric structure.
\qed
\end{remark}

\begin{remark}
The reader may recall that the Coulomb gauge fixing condition $\ast d^A \ast a + [\Phi,\phi]=0$ used to define the euclidean monopole metric has a natural interpretation: it says that $(a,\phi)$ is orthogonal to infinitesimal gauge transformations.  We do not know of a comparable interpretation for the gauge fixing condition \eqref{gauge fixing}.
\qed
\end{remark}

\section{Existence for gauge fixing condition}
\label{sec:existence}
\subsection{Proof strategy}
The purpose of this section is to prove the existence part of Theorems \ref{thm:gauge fixing} and \ref{thm:framed gauge fixing}. In other words, to show that the gauge-fixing condition \eqref{gauge fixing} has a solution satisfying the required properties.

We begin by reformulating \eqref{gauge fixing} in a 4-dimensional way.  Let $g_H=d\theta^2+g_{H^3}$ be the metric on $S^1\times H^3$ and let $\ast_H$ be the associated Hodge star.  Let $B=\Phi\,d\theta+A$ be a circle-invariant connection on $S^1\times H^3$ and let $b=\phi\,d\theta+a$ be the deformation of $B$ associated to $(a,\phi)$.  Recall that both $B$ and $b$ extend to $S^4$.
Then, for the choice of orientation $d  \theta \wedge \operatorname{Vol} _{ H ^3 }$, the gauge fixing condition \eqref{gauge fixing}  is equivalent to  
\begin{equation}
\label{gauge fixing s4} 
\delta^\pm  b=0,
\end{equation} 
 where
\begin{equation}\label{Delta}
\delta^\pm :\Omega^1(\su(2))\to\Omega^0(\su(2)),\quad\delta^\pm b=e^{\mp 4i\theta}\ast_H d^B\ast_H e^{\pm 2i\theta}b.
\end{equation}
Here the factor of $e^{\mp 4i\theta}$ is included for later notational convenience. We need to find an infinitesimal gauge transformation $\chi $ on $S^4$ such that $b-d^B\chi$ solves 
\begin{equation}
\label{gaugefixing4d} 
\delta^\pm (b-d^B\chi)=0
\end{equation}
on $S^4\setminus S^2$.

The way we solve (\ref{gaugefixing4d}) is in two steps. First we find $\chi_1$ such that $\delta^\pm (b-d^B\chi_1)$ equals $0$ on $S^2\subset S^4$ and decays rapidly as one approaches $S^2$. We do this in section \ref{solutionons2} and our argument relies on the fact that $H ^3 \times S ^1 $ is conformally related to the K\"ahler manifold $H ^2 \times S ^2 $. The second step  is to use elliptic theory to find $\chi_2$ solving $\delta^\pm (b-d^B\chi_1-d^B\chi_2)=0$ on $S^4$. We do this in  section \ref{solutionons4}, where we also show that the solution of the gauge fixing condition that we have constructed satisfies the properties listed in Theorems \ref{thm:gauge fixing}, \ref{thm:framed gauge fixing}. In the reminder of this section we describe the aforementioned conformal equivalence $H ^3 \times S ^1  \simeq H ^2 \times S ^2 $.

We have already seen in \eqref{conformal S1H3} that the metric $g_H$ on $S^1\times H^3$ is conformally equivalent to metric $g_S$ on $S^4\setminus S^2$.  Recall that we model $S^4$ as the set of vectors $y\in \R^5$ satisfying $y_\mu y_\mu=1$.  The metric $g_S$ is the restriction of the euclidean metric $dy_\mu dy_\mu$ to $S^4$.  Introducing $r_2^2=y_3^2+y_4^2+y_5^2$, we note that
\begin{equation}\label{conformal H2S2}
\frac{1}{r_2^2}dy_\mu dy_\mu = \frac{1}{r_2^2}(dr_2^2+dy_1^2+dy_2^2) + \frac{1}{r_2^2}(dy_3^2+dy_4^2+dy_5^2)
\end{equation}
is the product metric on $H^3\times S^2$, with $r_2>0, y_1, y_2$ being coordinates on $H^3$ and $S^2$ consisting of unit vectors $(y_3,y_4,y_5)/r_2$.  The $S^4$ constraint $y_\mu y_\mu=1$ is equivalent to $r_2^2+y_1^2+y_2^2=1$, which determines a submanifold of $H^3$ that is isometric to $H^2$.  So $S^4\setminus S^1$ is conformal to $H^2\times S^2$, with $S^1\subset S^4$ defined by $r_2=0$.

We let $g_K$ denote the metric on $H^2\times S^2$.  It is convenient to write it using the disc model of $H^2$.  To do so we introduce a stereographic coordinate $w\in\mathbb{C}$ satisfying $|w|<1$, such that
\begin{equation}
(r_2,y_1,y_2)=\left(\frac{1-|w|^2}{1+|w|^2},\frac{\bar{w}+w}{1+|w|^2},\frac{i\bar{w}-i{w}}{1+|w|^2}\right).
\end{equation}
The metric \eqref{conformal H2S2} on $H^2\times S^2$ then becomes
\begin{equation}\label{metric H2S2}
g_K = \frac{2}{(1-|w|^2)^2}(dw\,d{\bar w}+d{\bar w}d{w}) + \frac{2}{(1+|z|^2)^2}(dz\,d{\bar z}+d{\bar z}d{z}),
\end{equation}
with $z\in\C$ being a stereographic coordinate on $S^2$.  It is important to note that $g_K$ is K\"ahler, with K\"ahler form
\begin{equation}
\omega_K = \frac{2i}{(1-|w|^2)^2}dw\wedge d{\bar w} + \frac{2i}{(1+|z|^2)^2}dz\wedge d{\bar z}.
\end{equation}

We now have three conformally equivalent metrics $g_H,g_S,g_K$.  Their common domain of definition is $S^4\setminus(S^2\sqcup S^1)\simeq S^1\times (H^3\setminus\{0\}) \simeq (H^2\setminus\{0\})\times S^2$, with $0$ denoting points in the interior of $H^3$ and $H^2$.  From \eqref{conformal S1H3} and \eqref{conformal H2S2} we have  $g_H=r_1^{-2}g_S$ and $g_K=r_2^{-2} g_S$.  Since $r_2=(1-|w|^2)/(1+|w|^2)$ and $r_1^2=1-r_2^2$, the metric $g_H$ of equation \eqref{conformal S1H3} and the spherical metric $g_S$ are related to $g_K$ by
\begin{equation}\label{conformal equivalences}
g_H = \left(\frac{1+|w|^2}{2|w|}\right)^2g_S = \left(\frac{1-|w|^2}{2|w|}\right)^2 g_K.
\end{equation}
We note that the argument of the coordinate $w$ on $H^2$ equals the coordinate $\theta$ on $S^1$:  If we define $s:=-\ln|w|>0$ then we find from \eqref{conformal equivalences} and \eqref{metric H2S2} that
\begin{equation}
g_H = d\theta^2 + ds^2 + \sinh^2(s)\frac{2}{(1+|z|^2)^2}(dz\,d\bar{z}+d\bar{z}\,dz).
\end{equation}
So $s$ is a geodesic radial coordinate on $H^3$, i.e.~it measures the geodesic distance from the origin.  Therefore we can write $w=e^{-s+i\theta}$, with $s$ being a radial coordinate on $H^3$ and $\theta$ being the coordinate on $S^1$.

\subsection{Solution on $S^2$}
\label{solutionons2} 
The aim of this section is to  find an infinitesimal gauge transformation $\chi$ such that
\begin{equation}
\label{holomorphic gauge intro}
(b- d^B\chi )^{0,1} = w^{2p}\sigma ,
\end{equation}
for $\sigma \in\Gamma(\Lambda^{0,1}(S^2)\otimes \End_0(V)) $, and such that
\begin{equation}
\frac{1}{w^3}\delta^- (b- d^B\chi )
\end{equation}
is a smooth function on $S^4$. In particular  $(b- d^B\chi )^{0,1}$ and $\delta^- (b- d^B\chi ) $ vanish on the 2-sphere fixed by the circle action and corresponding to $w =0 $.

To do that, we first notice that $b ^{ 0,1 }$  determines a cohomology class in $H^{0,1}(H^2\times S^2,\End_0(V))$, with the subscript 0 indicating that $b^{0,1}$ is traceless.
In fact, the Bogomolony equation for $(A,\Phi)$ is equivalent to the anti-self-dual equations for $B$, which are conformally invariant and, in terms of the K\"ahler structure of $H^2\times S^2$, take the form
\begin{equation}\label{ASD}
(F^B)^{0,2}=0=(F^B)^{2,0},\quad\omega_K\wedge F^B=0.
\end{equation}
Their linearisation is
\begin{equation}\label{linearised ASD}
\bar{\partial}^Bb ^{0,1}=0={\partial}^Bb^{1,0},\quad\omega_K\wedge d^Bb=0.
\end{equation}
Thus $B$ determines a holomorphic structure on the vector bundle $V$ over $H^2\times S^2$ and $b^{0,1}$ is a $\bar{\partial}^B$-closed section of $\End_0(V)\otimes\Lambda^{0,1}$, which in turn determines a cohomology class in $H^{0,1}(H^2\times S^2,\End_0(V))$.  Hence in order to find a gauge where (\ref{holomorphic gauge intro}) is satisfied  it is useful to first determine the cohomology group $H^{0,1}(H^2\times S^2,\End_0(V))$.

We can determine the cohomology of $\End_0(V)$ using a natural filtration of $V$, i.e.\ a holomorphic line bundle $L\subset V$.  Following \cite{atiyah:1987}, we show that:
\begin{lemma}\label{lemma:L}
There is a unique circle-invariant holomorphic line bundle $L\subset V$ over $H^2\times S^2$ such that the restriction of $L$ to $\{0\}\times S^2$ equals the line bundle $L^-$ that appears in the monopole boundary conditions.  The degree of this bundle is $-n$.
\end{lemma}
\begin{proof}
The circle action \eqref{action on S^4} on $H^2\times S^2$ is $(w,z)\mapsto (\lambda w,z)$, with $\lambda=e^{it}$.  This lifts to an action of $S^1$ on the bundle $\mathbb{P}V$.  The action fixes the holomorphic structure of $\mathbb{P}V$, because the holomorphic structure is determined by the connection $B$ and $B$ is circle-invariant.  So the circle action on $\mathbb{P}V$ can be extended to a holomorphic action of the monoid $\mathbb{C}^\ast_{\leq 1}=\{\lambda\in\mathbb{C}^\ast\::\:|\lambda|\leq1\}$ (here the constraint on the modulus of $\lambda$ is imposed in order maintain the condition that $|w|<1$).

A line subbundle of $V$ corresponds to a section of $\mathbb{P}V$.  The line subbundle $L^-$ determines a section of $\mathbb{P}V$ over $\{0\}\times S^2$.  Let $C$ be the image of this section.  Then $C$ is a submanifold of $\mathbb{P}V$ whose points are fixed by the action of $\mathbb{C}^\ast_{\leq 1}$.  We will show that the unstable submanifold  $C ^u $ of $C$, which consists of points in $\mathbb{P}V$ that converge to $C$ under the action of $\lambda$ as $\lambda\to 0$, is the image of a section of $\mathbb{P}V$.  This section determines a circle-invariant holomorphic line bundle $L\subset V$.

First we determine the weights for the action of $S^1$ on the tangent space $T_c\mathbb{P}V$ at a point $c\in C$.  This point $c$ is in the fibre of $\mathbb{P}V$ above a point $(0,z)\in H^2\times S^2$, so there is a natural map $T_c\mathbb{P}V\to T_{(0,z)}H^2\times S^2$.  The kernel of this map is the tangent space of $\mathbb{P}(V_{(0,z)})$ at  $c$.  The point $c\in\mathbb{P}(V_{(0,z)})$ corresponds to the line $L^-_{(0,z)}\subset V_{(0,z)}$, and $V_{(0,z)}/L^-_{(0,z)}\cong L^+_{(0,z)}$, so the tangent space at $c$ is naturally isomorphic to $\Hom(L^-_{(0,z)},L^+_{(0,z)})$.
%Since $\mathbb{P}(V_{(0,z)}$ is a project space and its tangent space at $c$ is naturally isomorphic to $\Hom(L^-_{(0,z)},L^+_{(0,z)})$, because $c$ corresponds to the line $L^-_{(0,z)}\subset V_{(0,z)}$.
Thus, we have an exact sequence
\begin{equation}\label{weight exact sequence}
0 \to \Hom(L^-_{(0,z)},L^+_{(0,z)}) \to T_c\mathbb{P}V\to T_0H^2\oplus T_z S^2.
\end{equation}
From proposition \eqref{prop:S1 action}, the circle group acts on the first space with weight $-4p$, and it acts on the final two spaces with weights $2$ and $0$.  So it acts on $T_c\mathbb{P}V$ with weights $-4p$, $0$ and $2$.

We see that there are two non-negative weights, 2 and 0.  So the unstable manifold $C^u$ has complex dimension 2.  The tangent space of $C^u$ at $c\in C$ is the union of the weight 2 and 0 subspaces of $T_c\mathbb{P}V$, and this maps surjectively to $T_0H^2\oplus T_zS^2$ in \eqref{weight exact sequence}.

We need to show that the submanifold $C^u\subset \mathbb{P}V$ is the image of a section of $\mathbb{P}V\to H^2\times S^2$.  The inverse function theorem shows that this is true in an neighbourhood $\{w<\epsilon\}$ of $\{0\}\times S^2$, because the tangent space at $c\in C$ maps surjectively to the tangent space of $H^2\times S^2$.  So we have a section $\sigma:\{w<\epsilon\}\times S^2\to \mathbb{P}(V)$ whose graph is part of $C^u$.  Using the action of $\lambda=\epsilon$, we can extend this to a section $\sigma(w,z)=\epsilon^{-1}\cdot\sigma_\epsilon(\epsilon w,z)$ of $\mathbb{P}V$ on the whole of $H^2\times S^2$.  The image of this section is the submanifold $C^u$.

So $C^u$ is the graph of a section of $\mathbb{P}V$.  It therefore determines a line bundle $L\subset V$.  The degree of $L$ equals the degree of its restriction to $S^2$, which is $-n$.
\end{proof}
\begin{remark}
The line bundle $L\subset V$ can alternatively be described using monopole scattering \cite{atiyah:1987}.  Fix a point $z_0\in H^2$; this determines a radial line in $H^3$.  Recall that Hitchin's scattering equation along this line is 
\begin{equation}\label{scattering}
\left(\frac{d}{ds}+A_s-i\Phi\right)\sigma=0,
\end{equation}
in which $\sigma$ is a section of $V$ over the line and $s=-\ln|w|$ is the geodesic radial coordinate on $H^3$.  Then the fibre of $L$ above $(w_0,z_0)\in H^2\setminus\{0\}\times S^2$ is the set of $\sigma_0\in V_{w_0,z_0}$ such that the solution of the scattering equation with initial condition $\sigma=\sigma_0$ at $s=s_0:=-\ln|w_0|$ decays as $s\to\infty$.

To explain why this characterisation of $L$ is equivalent to that of Lemma \ref{lemma:L}, we note that if $\sigma$ is a solution of \eqref{scattering} and  satisfies $\partial_\theta\sigma=0$, then it satisfies $(\partial_s^B-i\partial_\theta^B)\sigma=0$ for the connection $B=\Phi\,d\theta+A$.  Since $w=e^{-s+i\theta}$, this means that $\sigma$ determines a circle-invariant holomorphic section of $V$ over $H^2\setminus\{0\}\times \{z_0\}$.  If $\sigma$ decays as $s\to\infty$ then it must approach the eigenspace of $\Phi|_{ S^2_\infty}$ with eigenvalue $ip$.  But this is precisely the fibre of $L^-$ over $z_0\in S^2_\infty$.  So $\sigma$ must be a section of the circle-invariant bundle $L|_{H^2\setminus\{0\}\times\{z_0\}}$ that extends $L^-_{z_0}$.
\qed
\end{remark}

The filtration $0\subset L\subset V$ determines a filtration of $\End_0(V)$
\begin{equation}
0=F_0\subset F_1\subset F_2 \subset F_3=\mathrm{End}_0(V),
\end{equation}
in which
\begin{align}
F_1&=\{\chi \in \mathrm{End}_0(V)\::\: \chi(V)\subset L,\,\chi(L)=0\} \\
F_2&=\{\chi \in \mathrm{End}_0(V)\::\: \chi(L)\subset L\}.
\end{align}
Then $F_i$ has rank $i$ and $F_1$, $F_2/F_1$ and $F_3/F_2$ are line bundles of degree $-n$, $0$ and $n$.  They are isomorphic to the standard bundles $\mathcal{O}(-n)$, $\mathcal{O}$ and $\mathcal{O}(n)$ over $S^2$, pulled back to $H^2\times S^2$.  These isomorphisms can be chosen so that $U(1)$ acts as $w\mapsto\lambda w$ on $H^2$ and by multiplication with $\lambda^{2p}$, $1$ and $\lambda^{-2p}$ on the fibres of $\mathcal{O}(-n)$, $\mathcal{O}$ and $\mathcal{O}(n)$ respectively.

With this notation in place, we are able to compute the Dolbeault cohomology of $\mathrm{End}_0(V)$.  The result that we need is
\begin{lemma}
The natural maps
\begin{multline}
H^{0,0}(H^2)\otimes H^{0,1}(S^2,\mathcal{O}(-n))\to  H^{0,1}(H^2\times S^2,F_1)\\
\to H^{0,1}(H^2\times S^2,\End_0(V))
\end{multline}
are surjective.
\end{lemma}
\begin{proof}
By the K\"unneth formula,
\begin{multline}
H^{0,1}(H^2\times S^2, F_3/F_2) \cong
H^{0,1}(H^2\times S^2,\mathcal{O}(n)) \\
\cong H^{0,0}(H^2)\otimes H^{0,1}( S^2,\mathcal{O}(n)) \oplus H^{0,1}(H^2)\otimes H^{0,0}( S^2,\mathcal{O}(n)).
\end{multline}
This vanishes, because $H^{0,1}( S^2,\mathcal{O}(n))=0$ for any 
$n>0$ and $H^{0,1}(H^2)=0$.  Therefore the map
\begin{equation}
H^{0,1}(H^2\times S^2,F_2)\to H^{0,1}(H^2\times S^2,F_3)
\end{equation}
associated with the exact sequence $F_2\to F_3\to F_3/F_2$ is surjective.  By a similar argument based on the fact that $H^{0,1}(S^2,\mathcal{O})=0$, the map
\begin{equation}
H^{0,1}(H^2\times S^2,F_1)\to H^{0,1}(H^2\times S^2,F_2)
\end{equation}
associated with the exact sequence $F_1\to F_2\to F_2/F_1$ is surjective.  Applying the K\"unneth formula once more, we obtain a natural isomorphism
\begin{align}
H^{0,1}(H^2\times S^2, F_1) &\cong
H^{0,1}(H^2\times S^2,\mathcal{O}(-n)) \\
%&\cong H^{0,0}(H^2)\otimes H^{0,1}( S^2,\mathcal{O}(-n)) \oplus H^{0,1}(H^2)\otimes H^{0,0}( S^2,\mathcal{O}(-n))\\
&= H^{0,0}(H^2)\otimes H^{0,1}( S^2,\mathcal{O}(-n))
\end{align}
since $H^{0,1}(H^2)=0$.
\end{proof}

Let now $\sigma_1,\sigma_2,\ldots$ be sections of $\Lambda^{0,1}(S^2)\otimes\mathcal{O}(-n)$ that represent a basis for $H^{0,1}(S^2,\mathcal{O}(-n))$.  By the preceding lemma  we can find an infinitesimal gauge transformation $\chi\in\Gamma(H^2\times S^2,\End_0(V))$ and holomorphic functions $f_i:H^2\to\C$ such that
\begin{equation}\label{fsigma}
{b}^{0,1}-\bar{\partial}^B\chi=\sum_i f_i\sigma_i.
\end{equation}
This can be done in such a way that ${b}-d^B\chi$ is still $U(1)$-invariant.  Recalling that $U(1)$ acts on $\mathcal{O}(-n)$ with weight $2p$, this means that
\begin{equation}
\sum_i f_i\sigma_i = \sum_i \lambda^{2p}f_i(\lambda^{-1}w)\sigma_i(u)\quad\forall \lambda \in U(1).
\end{equation}
Since the $\sigma_i$ are linearly independent, it follows that $\lambda^{2p}f_i(\lambda^{-1}w)=f_i(w)$ for each $i$.  The only holomorphic functions with this property are constant multiples of $w\mapsto w^{2p}$.  So the right hand side of \eqref{fsigma} is $w^{2p}\sigma$ for some linear combination $\sigma$ of the $\sigma_i$.  Therefore we are able to make an infinitesimal gauge transformation so that
\begin{equation}
\label{holomorphic gauge}
b^ {0,1}= w^{2p}\sigma ,\quad \sigma \in\Gamma(\Lambda^{0,1}(S^2)\otimes \End_0(V)),
\end{equation}
where, abusing notation, we have denoted by $b$ the gauge transformed deformation previously denoted by $b - d^B \chi $.

With this choice of gauge, $\delta^-b$ decays as $w\to 0$:
\begin{lemma}\label{lemma:smoothness}
Suppose that the gauge is chosen so that ${b}$ satisfies \eqref{holomorphic gauge} in a neighbourhood of $w=0$.  Then
\begin{equation}
\frac{1}{w^3}\delta^-b
\end{equation}
is a smooth function on $S^4$.
\end{lemma}
\begin{proof}
The operator $\delta^-$ of \eqref{Delta} can be rewritten using the second conformal equivalence of \eqref{conformal equivalences} as follows:
\begin{equation}\label{codifferential}
\delta^- = \left(\frac{2 w}{1-|w|^2}\right)^4\ast_K d^B \ast_K \left(\frac{1-|w|^2}{2w}\right)^2.
\end{equation}
If ${b}$ satisfies \eqref{holomorphic gauge} then this simplifies considerably.  To see why, note that \eqref{codifferential} implies that
\begin{equation}
\label{codifferentialcomputation} 
\delta^-b=\left(\frac{2 w}{1-|w|^2}\right)^2\ast_K d^B \ast_K {b} - 
\left(\frac{2w}{1-|w|^2}\right)^3 \left(\frac{dw}{w^2}+d\bar{w}\right)\lrcorner_K b.
\end{equation}
The contraction with $dw$ is $0$ in the gauge \eqref{holomorphic gauge} because $b_{\bar w}=0$ and $dw\lrcorner dw=0$.  So the second term is $w^3$ times a smooth function.  The first term can be simplified using the linearised anti-self-dual equations \eqref{linearised ASD} and the fact that $d\omega_K=0$:
\begin{equation} 
\begin{split} 
d^B\ast_K{b}&
=d^B (i\omega_K\wedge({b}^{0,1}-{b}^{1,0}))\\
&=d^B(\omega_K\wedge (2i{b}^{0,1}-i{b}))\\
&= 2i\omega_K\wedge d^B{b}^{0,1}\\
&=2i\omega_K\wedge (2pw^{2p-1}dw\wedge \sigma +w^{2p}d^B \sigma )\\
&=2iw^{2p}\omega_K\wedge d^B \sigma  .
\end{split} 
\end{equation} 
Since $2p\geq 1$, the first term in \eqref{codifferentialcomputation} is also $w^3$ times a smooth function.
\end{proof}

\subsection{Solution on $S^4$}
\label{solutionons4} 
Having found a gauge where \eqref{holomorphic gauge} holds, we now try to solve the gauge fixing condition (\ref{gauge fixing s4}) on $S ^4 $.
%We have shown that, after an infinitesimal gauge transformation, the circle-invariant solution $b$ of the linearised anti-self-dual equation satisfies \eqref{holomorphic gauge} in a neighbourhood of $S^2\subset S^4$.}  Having chosen this initial gauge, we now try to solve the gauge fixing condition \eqref{gauge fixing.
In other words, we seek a section $\chi$ of $\operatorname{End}_0(V)\to S^4$ such that
\begin{equation}\label{gf3}
\delta^-(b+d^B\chi)=0,
\end{equation}
where $\delta^-$ is given in \eqref{Delta}.  We rewrite (\ref{gf3})  using the following result
\begin{lemma}\label{lemma:Hodge}
Let $\triangle_S^B=-\ast_S d^B\ast_S d^B$ denote the (positive) Hodge laplacian associated with the connection $B$ over $S^4$.  For any section $\chi$ of $\operatorname{End}_0(V)$,
\begin{equation}
\delta^-d^B\chi = -\left(\frac{2w}{1+|w|^2}\right)^3(\triangle^B_S+2)\left(\frac{1+|w|^2}{2w}\right)\chi.
\end{equation}
\end{lemma}
\begin{proof}
The first conformal equivalence of \eqref{conformal equivalences} implies that
\begin{equation}\label{eq1}
\delta^-d^B\chi = \left(\frac{2w}{1+|w|^2}\right)^4\ast_S d^B \ast_S \left(\frac{1+|w|^2}{2w}\right)^2 d^B\chi.
\end{equation}
The statement of the lemma is therefore equivalent to
\begin{equation}\label{identity}
\ast _Sd^B \ast _S\left(\frac{1+|w|^2}{2w}\right)^2 d^B\chi = 
-\left(\frac{1+|w|^2}{2w}\right)(\triangle^B_S+2)\left(\frac{1+|w|^2}{2w}\right)\chi.
\end{equation}
We prove the identity \eqref{identity} by expanding both sides.  The left hand side is clearly equal to
\begin{equation}
-\left( \frac{1+|w|^2}{2w}\right)^2\triangle_S^B\chi + 2\left( \frac{1+|w|^2}{2w}\right)\ast_S\left(d\left(\frac{1+|w|^2}{2w}\right)\wedge\ast_S \, d^B\chi\right).
\end{equation}
The right hand side equals
\begin{multline}
-   \left( \frac{1+|w|^2}{2w}\right)^2\triangle_S^B\chi + 2\left( \frac{1+|w|^2}{2w}\right)\ast_S\left(d\left(\frac{1+|w|^2}{2w}\right)\wedge\ast_S \, d^B\chi\right)\\
- \left( \frac{1+|w|^2}{2w}\right)\chi (\triangle_S+2)\left( \frac{1+|w|^2}{2w}\right).
\end{multline}
A direct calculation shows that
\begin{equation}
    \triangle_S\left( \frac{1+|w|^2}{2w}\right)=-2\left( \frac{1+|w|^2}{2w}\right),
\end{equation}
so the two sides of the identity agree.
\end{proof}
Lemma \ref{lemma:Hodge} reduces the gauge fixing condition \eqref{gf3} to
\begin{equation}\label{poisson equation}
(\triangle_S^B+2)\left(\frac{1+|w|^2}{2w}\right)\chi = \left(\frac{1+|w|^2}{2w}\right)^3\delta^-b.
\end{equation}
Lemma \ref{lemma:smoothness} shows that, since $b$ satisfies \eqref{holomorphic gauge}, the right hand side is a smooth function near $S^2\subset S^4$.  Standard theory then guarantees that \eqref{poisson equation} admits a unique smooth solution.  In more detail, the operator $\triangle_S^B+2$ is a self-adjoint second order elliptic operator.  Its kernel is 0, because
\begin{equation}
(\triangle^B_S+2)f=0 \quad\implies\quad \int_{S^4} \left( |d^Bf|^2+2|f|^2\right)  \mathrm{Vol} _{ S  } =0 \quad\implies\quad f=0.
\end{equation}
By the Fredholm alternative, \eqref{poisson equation} admits a unique solution $\chi$.  By uniqueness, this solution is circle-invariant.  So $b_-=b+d^B\chi$ solves the gauge-fixing condition \eqref{gf3} and is circle-invariant.  It gives a solution $(a_-,\phi_-)$ of the gauge-fixing condition \eqref{gauge fixing} with a minus sign.  
%This proves the existence part of Theorem \ref{thm:gauge fixing}.

To complete the proof of the existence part of Theorems \ref{thm:gauge fixing}, \ref{thm:framed gauge fixing} we still need to show that  $a _- $ is a section of $\Lambda ^{ 1,0 }$ over $S ^2 _\infty $ and that the $L ^2 $ norm of $(a _- , \phi _- )$ is finite. Let us start with the  behaviour over $S ^2 _\infty $.
\begin{lemma}\label{lemma:BC b}
The solution $(a_-,\phi_-)$ of \eqref{gauge fixing} in Theorems \ref{thm:gauge fixing} and \ref{thm:framed gauge fixing} is such that the (0,1)-part of $b_-=a_-+\phi_-\,d\theta$ is equal to $w$ times a smooth section near $S^2\subset S^4$.
\end{lemma}
\begin{proof}
From \eqref{holomorphic gauge} we see that $b^{0,1}$ is equal to $w$ times a smooth section near $w=0$.  From \eqref{poisson equation} we see that $\chi$ is also equal to $w$ times a smooth section near $w=0$.  Since $w$ is a holomorphic coordinate, the $(0,1)$ part of $d^B\chi$ is also equal to $w$ times a smooth section, and so the $(0,1) $ part of $b_-=b + d^B\chi$ equals $w$ times a smooth section. This proves the Lemma in the case of Theorem \ref{thm:gauge fixing}.  For Theorem \ref{thm:framed gauge fixing} we  need in addition to show that the (0,1)-part of $d^A\Phi$ is $w$ times a smooth section.  The connection $B=A+\Phi\,d\theta$ has curvature $F^B=F^A+d^A\Phi\wedge d\theta$, so
\begin{equation}
d^A\Phi = -\frac{\partial}{\partial\theta}\lrcorner F^B = -iw\frac{\partial}{\partial w}\lrcorner F^B + i\bar{w}{\frac{\partial}{\partial \bar{w}}}\lrcorner F^B.
\end{equation}
The anti-self-dual equations \eqref{ASD} imply that $F^B$ is of type (1,1), so the (0,1)-part of $d^A\Phi$ is $-i{w}{\frac{\partial}{\partial {w}}}\lrcorner F^B$.  Thus $d^A\Phi$ is $w$ times a smooth section.
\end{proof}

%Next we investigate the behaviour of the solution near $S^2$.  From \eqref{holomorphic gauge} we see that $b^{0,1}$ is equal to $w$ times a smooth section near $w=0$.  From \eqref{poisson equation} we see that $\chi$ is also equal to $w$ times a smooth section near $w=0$.  Since $w$ is a holomorphic coordinate, the $(0,1)$ part of $d^B\chi$ is also equal to $w$ times a smooth section, and so the $(0,1) $ part of $b_-=b + d^B\chi$ equals $w$ times a smooth section.  This proves part of the following lemma:
%\begin{lemma}\label{lemma:BC b}
%The solution $(a_-,\phi_-)$ of \eqref{gauge fixing} in Theorems \ref{thm:gauge fixing} and \ref{thm:framed gauge fixing} is such that the (0,1)-part of $b_-=a_-+\phi_-\,d\theta$ is equal to $w$ times a smooth section near $S^2\subset S^4$.
%\end{lemma}
%\begin{proof}
%The statement is already proved in the case of Theorem \ref{thm:gauge fixing}.  For Theorem \ref{thm:framed gauge fixing} we in addition need to show that the (0,1)-part of $d^A\Phi$ is $w$ times a smooth section.  The connection $B=A+\Phi\,d\theta$ has curvature $F^B=F^A+d^A\Phi\wedge d\theta$, so
%\begin{equation}
%d^A\Phi = -\frac{\partial}{\partial\theta}\lrcorner F^B = -iw\frac{\partial}{\partial w}\lrcorner F^B + i\bar{w}{\frac{\partial}{\partial \bar{w}}}\lrcorner F^B.
%\end{equation}
%The anti-self-dual equations \eqref{ASD} imply that $F^B$ is of type (1,1), so the (0,1)-part of $d^A\Phi$ is $-i{w}{\frac{\partial}{\partial {w}}}\lrcorner F^B$.  Thus $d^A\Phi$ is $w$ times a smooth section.
%\end{proof}

Finally, we show that the solution $(a_-,\phi_-)$ of the gauge fixing condition \eqref{gauge fixing} has finite $L ^2 $ norm. The $L ^2 $  integral is proportional to
%Now we show that the metric integral \eqref{metric integral} is finite for the solution $(a_-,\phi_-)$ of the gauge fixing condition \eqref{gauge fixing}.  This  integral is proportional to
\begin{align}\label{integral 1}
-\int_{S^4\setminus S^2}\tr(b_-\wedge\ast_H b_-) = -\int_{S^4\setminus S^2}\left(\frac{1+|w|^2}{2|w|}\right)^2\tr(b_-\wedge\ast_S b_-).
\end{align}
The integrand is continuous away from $w=0$, so the integral will be finite as long as the integral in a neighbourhood of $0$ is finite.  The integral on a neighbourhood $0<|w|<\epsilon$ can be written using the K\"ahler metric \eqref{metric H2S2} as follows:
\begin{equation}\label{integral 2}
-\int_{0<|w|<\epsilon}\tr(b_-\wedge\ast_H b_-) = -2\int_{0<|w|<\epsilon}\left(\frac{1-|w|^2}{2|w|}\right)^2 \tr(b_-^{1,0}\wedge\ast_K b_-^{0,1}).
\end{equation}
By Lemma \ref{lemma:BC b}, $b_-^{0,1}$ is $w$ times a smooth section, so the integrand is $dw\wedge d\bar{w}/\bar{w}$ times the volume form of $S^2$ times a function that extends smoothly over $w=0$.  Since $\int_{0<|w|<\epsilon}dw\wedge d\bar{w}/|w|$ converges, the integral \eqref{integral 2} converges.  This completes the proof of Theorem \ref{thm:gauge fixing}.

\section{Killing spinors and pluricomplex structures}
\label{sec:Killing}
The gauge fixing condition \eqref{gauge fixing} was motivated by a connection with Killing spinors \cite{figueroa-ofarrill:2014}.  We will present a precise formulation of this correspondence below, and show that it induces various geometric structures on the framed moduli space.  We then show that these geometric structures are compatible with the symmetric bilinear forms $g_\pm$.

\subsection{Killing spinors and Dirac operators}

A Killing spinor on a riemannian manifold $M$ with a spin structure is a section $\psi$ of the spinor bundle $S$ that solves
\begin{equation}
\nabla_X\psi = \lambda X\cdot\psi
\end{equation}
for all tangent vectors $X$.  Here $\lambda\in\C$ is known as the Killing constant of $\psi$, and the dot denotes Clifford multiplication.

To relate this to monopoles, let us suppose that $M$ is three-dimensional.  Then the Clifford algebra admits two inequivalent irreducible representations; we choose the one in which the volume form acts as the identity, and this choice determines the spinor bundle  $S$. We recall that $S$  is a complex rank 2 bundle admitting a non-degenerate skew-symmetric pairing $(\cdot,\cdot)$. Such a pairing induces an isomorphism $S\cong S^\ast$, written as $s\mapsto s^\ast$, such that the pairing is $(s,t)=s^\ast(t)$.  In a suitable local frame the map $s \mapsto s^\ast $  takes the form
\begin{equation}
\begin{pmatrix}s_1\\s_2\end{pmatrix}^\ast = \begin{pmatrix}-s_2 & s_1\end{pmatrix},
\end{equation}
corresponding to the skew-symmetric pairing 
\begin{equation}
\label{sympl pairing} 
(s,t ) = \begin{pmatrix}s_1 & s_2\end{pmatrix} \begin{pmatrix}
0 & 1 \\
-1 & 0
\end{pmatrix} \begin{pmatrix}
t _1  \\
t _2 
\end{pmatrix} = s _1 t _2 - s _2 t _1 .
\end{equation} 
The skew-symmetric pairing is compatible with Clifford multiplication in the sense that $(s,v\cdot t)=-(v\cdot s,t)$ and $(v\cdot s)^\ast=-s^\ast\cdot v$ for all $v\in\Lambda^1$.

Clifford multiplication induces an isomorphism
\begin{equation}\label{clifford isomorphism}
\cl_H: (\Lambda^0\oplus\Lambda^1)\otimes\C \to \End(S)\cong S\otimes S^\ast\cong S\otimes S
\end{equation}
such that $(a  + \phi ) \cdot \psi=\cl_H (a,\phi ) \psi $ for all  $a \in \Lambda ^1 , \phi \in \Lambda ^0 , \psi \in S $.
Under this isomorphism $\Lambda^1$ corresponds to the traceless endomorphisms $\End_0(S)$ and $\Lambda^0$ corresponds to the identity component.

Three-dimensional hyperbolic space $H^3$ admits Killing spinors with Killing constant $\pm\frac{i}{2}$.  The space $K^\pm$ of Killing spinors with Killing constant $\pm\frac{i}{2}$ is two-dimensional \cite{fujii:1986}.  Since the spinor bundle $S$ is  two-dimensional and Killing spinors are determined by their value at a point,  Killing spinors trivialise $S$.  We thus have the following natural  isomorphisms of vector bundles, for $p \in H ^3 $, $\psi ^\pm \in K ^\pm $,
\begin{equation}
K^{\pm}\times H^3 \cong S, \quad (\psi ^\pm , p ) \mapsto \psi ^\pm (p) .
\end{equation}
It follows that
\begin{equation}
\Gamma(\Lambda^0\oplus\Lambda^1)\otimes\C \cong \Gamma(S)\otimes K^\pm\label{spinor isomorphism}.
\end{equation}
More explicitly, this means that given a pair of independent Killing spinors $\psi_1,\psi_2$ in $K^-$,  any section $(a,\phi)$ of $(\Lambda^1\oplus\Lambda^0)\otimes\C$ can be written in the form $(a,\phi)=\cl_H^{-1}(\nu_1\otimes\psi_1^\ast +\nu_2\otimes\psi_2^\ast)$ for a unique pair of sections $\nu_1,\nu_2$ of $S$ (and similarly for $K^+$).

\begin{proposition}\label{prop:Killing}
Let $(A,\Phi)$ be a hyperbolic monopole and let $D^A:\Gamma(S\otimes\End(V))\to\Gamma(S\otimes\End(V))$ be the Dirac operator associated to $A$.  Under the isomorphism \eqref{spinor isomorphism}, we have the identity
\begin{equation}
\left(-D^A+\ad(\Phi)\pm\frac{i}{2}\right)\otimes\mathrm{id}_{K^\pm} = 
\begin{pmatrix} * d^A+\ad(\Phi) & -d^A \\ * d^A* & \ad(\Phi)\pm 2i \end{pmatrix},
\end{equation}
in which the operator on the right is written with respect to the splitting $\Lambda^1\oplus\Lambda^0$.
\end{proposition}
\begin{proof}
Choose an orthonormal frame $e^i$ for the cotangent bundle, and choose a basis $\psi_1,\psi_2$ for $K^\pm$.  To simplify notation, write $\lambda=\pm\frac{i}{2}$.  Then for any sections $\nu_1,\nu_2$ of $S\otimes\End(V)$,
\begin{align}
e^i\cdot\nabla^A_i(\nu_\alpha\otimes\psi^\ast_\alpha) &= e^i\cdot(\nabla^A_i\nu_\alpha)\otimes\psi^\ast_\alpha + \lambda e^i\cdot\nu_\alpha\otimes(e^i\cdot\psi_\alpha)^\ast \\
&=(D^A\nu_\alpha)\otimes\psi_\alpha^\ast - \lambda e^i\cdot \nu_\alpha\otimes\psi_\alpha^\ast \cdot e^i \\
&= (D^A\nu_\alpha-\lambda\nu_\alpha)\otimes\psi_\alpha^\ast + 2\lambda\tr_S(\nu_\alpha\otimes\psi_\alpha^\ast),
\end{align}
in which we used the identity $e^i\cdot x\cdot e^i = x-2\tr_S(x)$ for $x\in \mathrm{End}(S)$.  Writing $\nu_\alpha\otimes\psi ^\ast _\alpha = a+\phi$ for sections $a,\phi$ of $\Lambda^1\otimes\End(V)$ and $\Lambda^0\otimes\End(V)$, it follows that
\begin{equation}
(-D^A\nu_\alpha+[\Phi,\nu_\alpha]+\lambda\nu_\alpha)\otimes\psi_\alpha^\ast = -e^i\cdot\nabla^A_i(a+\phi)+[\Phi,a+\phi]+4\lambda\phi,\label{killing proof 1}
\end{equation}
having used $\operatorname{Tr}_S  (\nu _\alpha \otimes \psi ^\ast  _\alpha )= \operatorname{Tr} _S(  \phi I _2 + a ^i e _i \cdot  ) = 2  \phi $.
By direct calculation,
\begin{equation}
-e^i\cdot\nabla^A_i({\phi}+a) = -e^i\nabla_i^A{\phi} - e^i\wedge \nabla_i^Aa +e^i\lrcorner \nabla_i^A{a} = -d^A{\phi} - d^A{a} + * d^A* {a}.\label{killing proof 2}
\end{equation}
Since we have chosen the  Clifford algebra representation in which the volume form $\mathrm{Vol}$ acts as multiplication by 1,
\begin{equation}
-d^A{a} = -\mathrm{Vol}\cdot d^A{a} = * d^A{a}.\label{killing proof 3}
\end{equation}
The result follows from equations \eqref{killing proof 1}, \eqref{killing proof 2} and \eqref{killing proof 3}.
\end{proof}
\subsection{Eigenspinors and tangent vectors}
Proposition \ref{prop:Killing} shows that tangent vectors $(a,\phi)$ to the moduli space of hyperbolic monopoles solving \eqref{linearised bog_eq}, \eqref{gauge fixing} give rise to eigenspinors $\nu$ solving
\begin{equation}\label{eigenspinors}
D^A\nu -[\Phi,\nu]=\pm\frac{i}{2}\nu.
\end{equation}
We would like to establish a precise relationship between tangent vectors and eigenspinors.  To do this we must impose boundary conditions on the solutions $\nu$ of \eqref{eigenspinors}.  We will do so by lifting the spinors to $S^4$.

Recall that the spinor bundle of $S^1\times H^3$ takes the form $S_+\oplus S_-$, in which $S_\pm$ are eigenspaces of the action of the volume form by Clifford multiplication with eigenvalues $\mp1$.  This means that self-dual two-forms act by zero on $S_-$, while anti-self-dual 2-forms act by zero on $S_+$.  Both $S_+$ and $S_-$ are isomorphic to the pull-back  of the spinor bundle $S$ of $H^3$.  The Dirac operator associated to the connection $\Phi\,d\theta+A$ splits as the sum of an operator $D^B_H:\Gamma(S^1\times H^3,S_-\otimes\End_0(V))\to \Gamma(S^1\times H^3,S_+\otimes\End_0(V))$ and its adjoint. We focus on $D^B_H $, which takes the form 
\begin{equation}
D^B_H = D^A-\frac{\partial}{\partial\theta}-\ad(\Phi).
\end{equation}
So the eigenspinor equation \eqref{eigenspinors} is equivalent to
\begin{equation}
D^B_H (e^{\pm \frac{i\theta}{2}}\nu )= 0.
\end{equation}
Recall that the metrics $g_S$ of $S^4$ and $g_H$ of $S^1\times H^3$ are conformally equivalent \eqref{conformal equivalences}.  It follows that their Dirac operators are related by 
\begin{equation}
D^B_H=\left(\frac{2|w|}{1+|w|^2}\right)^{\frac52} D^B_S \left(\frac{2|w|}{1+|w|^2}\right)^{-\frac32}.
\end{equation}
For $ x \in M ^f _{ n,p }$ we define $E_x^\pm$ to be  the space of solutions of \eqref{eigenspinors} such that
\begin{equation}\label{tilde nu}
\tilde{\nu}:= \left(\frac{2|w|}{1+|w|^2}\right)^{-\frac32}  \nu e^{\pm \frac{i\theta}{2}}
\end{equation}
extends smoothly to $S^4$.  Then $E_x^\pm$ is naturally isomorphic to the kernel of $D^B_S:\Gamma(S^4,S_-\otimes\End_0(V))\to\Gamma(S^4,S_+\otimes\End_0(V))$.

By construction, the spinor $\tilde{\nu}$ transforms with weight $\pm  1/2$ under the circle action.  This fact allows us to determine the dimension of $E_x^\pm$ using the equivariant index theorem:
\begin{proposition}
The complex vector spaces $E_x^\pm$ have dimension $2n$.
\end{proposition}
\begin{proof}
Let $U(1)$ be the double-cover of the circle group acting on $S^4$.  For $\gamma\in U(1)$, we define 
\begin{equation}\label{index definition}
\text{ind}_{U(1)}(\gamma,D^B_S) = \tr(\gamma,\ker D^B_S) - \tr(\gamma,\coker D^B_S).
\end{equation}
The equivariant index theorem says that
\begin{equation}
\text{ind}_{U(1)}(\gamma,D^B_S) = \frac{1}{2\pi i}\int_{S^2} \frac{\hat{A}(S^2)\ \text{ch}_{U(1)}(\gamma,\End_0(V))}{\text{ch}_{U(1)}(\gamma,S^+(NS^2))-\text{ch}_{U(1)}(\gamma,S^-(NS^2))}.
\end{equation}
Here $S(NS^2)=S^+(NS^2)\oplus S^-(NS^2)$ is the spinor bundle associated with the normal bundle of $S^2\subset S^4$.  The equivariant Chern character of a bundle $W\to S^2$ is simply
\begin{equation}
\text{ch}_{U(1)}(\gamma,W):= \tr_E(\gamma\cdot \exp(-F)),
\end{equation}
in which $F$ is the curvature of a connection on $W$.  The $\hat{A}$ genus of $S^2$ is given in terms of its curvature tensor $R$ simply by
\begin{equation}
\hat{A}(S^2)=\left(\det\frac{R/2}{\sinh R/2}\right)^{\frac{1}{2}}=1-\frac{1}{48}\tr_{TS^2}(R^2) +\ldots = 1.
\end{equation}

To calculate the equivariant Chern character, recall that $V\to S^2$ is isomorphic to $L^+\oplus L^-$ with $L^-=(L^+)^\ast$, and so the bundle $\End_0(V)\to S^2$ is isomorphic to $(L^+)^{\otimes 2}\oplus \C \oplus (L^-)^{\otimes 2}$.  The degrees of these bundles are $2n,0,-2n$ and their weights under the $U(1)$ action are $-4p,0,4p$.  We choose a connection on $L^+$ with curvature $-i\omega/2$, where $\omega$ is the area form on $S^2$ whose integral is $4\pi$.  Then
\begin{equation}
\text{ch}_{U(1)}(\gamma,\End_0(V)) = \gamma^{-4p}+1+\gamma^{4p}+in\omega(\gamma^{4p}-\gamma^{-4p}).
\end{equation}

The normal bundle $N S^2$ is topologically trivial, and so $S^\pm (NS^2)$ are also trivial.  The action of the circle group on $NS^2$ has weights $\pm1$, and the action of its double cover $U(1)$ has weights $\pm2$, so the action of $U(1)$ on the positive and negative spinor bundles has weights $\pm1$.  Thus:
\begin{equation}
\text{ch}_{U(1)}(\gamma,S^\pm(NS^2))=\gamma^{\pm 1}.
\end{equation}
Altogether, we have that
\begin{equation}
\text{ind}_{U(1)}(\gamma,D^B_S) = \frac{n}{2\pi}\int_{S^2}\frac{(\gamma^{4p}-\gamma^{-4p})}{\gamma-\gamma^{-1}}\omega 
= 2n\sum_{j=0}^{4p-1}\gamma^{2j+1-4p}.\label{index calculation}
\end{equation}

The spaces $E_x^\pm$ are the weight spaces in $\ker D^B_S$ with weight $\pm 1$ under the double cover $U(1)$ of the circle group.  Since the connection $B$ is anti-self-dual, a well-known argument shows that the cokernel of $D^{B}_S$ is trivial \cite{atiyah:1978}.  Therefore the equivariant index \eqref{index definition} is equal to the character of the representation $\ker(D^{B}_S)$ of $U(1)$.  The dimensions of the weight spaces $E^\pm_x$ are the coefficients $\gamma^{\pm1}$ in this polynomial, and from the calculation \eqref{index calculation} they are both equal to $2n$ as claimed.
\end{proof}

We now aim to relate $E_x^\pm$ to the tangent space of the framed monopole moduli space.  Let $\nu\in E_x^\pm$ and let $\psi\in K^\pm$.  Then $\nu\otimes\psi$ gives a section $(a_\pm ,\phi_\pm)$ of $(\Lambda^1H^3\oplus\Lambda^0H^3)\otimes\End_0(V)$ via the isomorphism \eqref{clifford isomorphism}.  Proposition \ref{prop:Killing} shows that $(a_\pm,\phi_\pm)$ solves the linearised Bogomolny equation \eqref{linearised bog_eq} and the gauge-fixing condition \eqref{gauge fixing}.  To show that this is a tangent vector to the moduli space we must show the following:
\begin{proposition}\label{prop:b- extends}
If $\nu\in E_x^\pm$, $\psi\in K^\pm$ and $(a_\pm,\phi_\pm):=\cl_H^{-1}(\nu\otimes\psi)$, then $b_\pm:=\phi_\pm\,d\theta+a_\pm$ extends over $S^2$ to give a smooth section of $\Lambda^1S^4\otimes\End_0(V)$.
\end{proposition}
\begin{proof}
For notational simplicity we prove the proposition only for $b_-$.  The proof for $b_+$ is similar.  Using the Clifford isomorphism
\begin{equation}\label{clifford isomorphism sphere}
\cl_S:\Lambda^1 S^4 \to \Hom(S_+,S_-)\cong S_-\otimes S_+^\ast \cong S_-\otimes S_+
\end{equation}
induced by the metric $g_S$ on $S^4$, $b_-$ can be identified with the section $\cl_S(b_-)$ of $S_-\otimes S_+\otimes \End_0(V)$, defined away from $S^2\subset S^4$.  The Clifford isomorphisms \eqref{clifford isomorphism}, \eqref{clifford isomorphism sphere} are related by
\begin{equation}\label{two clifford isomorphisms}
\cl_S =\left(  \frac{1+|w|^2}{2|w|}\right)  \cl_H,
\end{equation}
in which we have implicitly identified $\Lambda^1(S^4\setminus S^2)$ with $\Lambda^1(H^3)\oplus \Lambda^0(H^3)$ and identified $S$ with $S_+$ and $S_-$.  So $\cl_S(b_-)$ is given by
\begin{equation}\label{b nu psi}
\cl_S(b_-)=\left(  \frac{1+|w|^2}{2|w|}\right)  \cl_H(a_-,\phi_-)=\left( \frac{1+|w|^2}{2|w|} \right) \nu\otimes\psi,
\end{equation}
in which $\nu$ and $\psi$ are interpreted as sections of $S_-\otimes\End_0(V)$ and $S_+$.  The section \eqref{b nu psi} can be written $\tilde{\nu}\otimes\tilde\psi$, with $\tilde{\nu}$ given in \eqref{tilde nu} and
\begin{equation}\label{tilde psi}
\tilde{\psi}:=\left(\frac{2|w|}{1+|w|^2}\right)^{\frac12}e^{\frac{i\theta}{2}}\psi.
\end{equation}
Since $\tilde{\nu}$ is a smooth section over $S^4$, $b$ will extend smoothly over $S^4$ if $\tilde{\psi}$ extends to a smooth section of $S_+\to S^4$. The simplest way to show that $\tilde{\psi}$ extends smoothly is by explicitly solving the Killing spinor equation.

We do so in the upper half-space model of $H^3$.  Recall that $S^4$ is the set of unit vectors $y\in\mathbb{R}^5$.  We define coordinates $\theta,\rho>0,x,y$ on $S^4\setminus S^2\simeq S^1\times H^3$ as follows:
\begin{equation}
(y_1,y_2,y_3,y_4,y_5)=\frac{1}{1 + \rho^2+x^2+y^2}(2\rho\cos\theta,2\rho\sin\theta,2x,2y,1-\rho^2-x^2-y^2).
\end{equation}
Then the conformal factor in \eqref{conformal equivalences} is $2|w|/(1+|w|^2)=r_1=2\rho/ (1+\rho^2+x^2+y^2)$.  The hyperbolic metric and volume form are
\begin{equation}
g = \frac{1}{\rho^2}(d\rho^2+dx^2+dy^2),\quad \mathrm{Vol}=-\frac{1}{\rho^3}d\rho\wedge dx\wedge dy.
\end{equation}
We choose a frame for the spinor bundle $S\to H^3$ so that the Clifford isomorphism \eqref{clifford isomorphism} is given by Pauli matrices:
\begin{equation}\label{Clifford representation}
\cl_H:\frac{d\rho}{\rho}\mapsto - i\sigma_3,\quad\frac{dx}{\rho}\mapsto - i\sigma_1,\quad\frac{dy}{\rho}\mapsto - i\sigma_2.
\end{equation}
This ensures that $\cl_H(\mathrm{Vol})$ is the identity matrix $I_2$.  Then the Levi-Civita connection on the spinor bundle is
\begin{equation}\label{Levi Civita connection}
\nabla = d + \frac{dx}{2\rho}i\sigma_2 - \frac{dy}{2\rho}i\sigma_1.
\end{equation}

A quick proof that this is the Levi-Civita connection is as follows: we compute
\begin{equation}
[\nabla,i\sigma_1]=\frac{dx}{\rho}i\sigma_3,\quad
[\nabla,i\sigma_2]=\frac{dy}{\rho}i\sigma_3,\quad
[\nabla,i\sigma_3]=-\frac{dx}{\rho}i\sigma_1-\frac{dy}{\rho}i\sigma_2.
\end{equation}
So
\begin{equation}
\nabla\frac{dx}{\rho}=\frac{dx}{\rho}\otimes\frac{d\rho}{\rho},\quad
\nabla\frac{dy}{\rho}=\frac{dy}{\rho}\otimes\frac{d\rho}{\rho},\quad
\nabla\frac{d\rho}{\rho}=-\frac{dx}{\rho}\otimes\frac{dx}{\rho}-\frac{dy}{\rho}\otimes\frac{dy}{\rho}.
\end{equation}
It is straightforward to check that this connection on $\Lambda^1( \mathbb{H}^3)$ is orthogonal.  It is torsion-free because $\nabla\wedge v=dv$ for $v=dx/\rho,dy/\rho,d\rho/\rho$.  It is straightforward to check that 
\begin{equation}\label{explicit Killing spinors}
\psi_1=\rho^{-\frac12}\begin{pmatrix}1\\0\end{pmatrix},\quad\psi_2=\rho^{-\frac12}\begin{pmatrix} -x+iy \\ \rho\end{pmatrix}
\end{equation}
are Killing spinors with Killing constant $-\frac{\ii}{2}$ which form a basis for $K^-$.

Now we aim to write $\psi_\alpha$ as a spinors on $S^4\setminus S^2$.  We identify $S^1\times H^3\simeq S^4\setminus S^2$ with $\R^4\setminus\R^2$.  The map $S^1\times H^3\to \R^4\setminus\R^2$ is given by $(\theta,\rho,x,y)\mapsto(\rho\cos\theta,\rho\sin\theta,x,y)$.  The map $S^4\setminus S^2\to\R^4\setminus\R^2$ is stereographic projection.  The metric $g_S$ is then $4/(1+\rho^2+x^2+y^2)^2$ times the euclidean metric.  We can choose frames for $S_+,S_-\to S^4$ on the whole of $\R^4$ such that the Clifford isomorphism \eqref{clifford isomorphism sphere} is given by
\begin{equation}\label{Clifford representation sphere}
\begin{aligned}
\cl_S:\frac{2d(\rho\cos\theta)}{1+\rho^2+x^2+y^2}&\mapsto - i\sigma_3, & 
\frac{2d(\rho\sin\theta)}{1+\rho^2+x^2+y^2}&\mapsto I_2,\\
\frac{2dx}{1+\rho^2+x^2+y^2}&\mapsto - i\sigma_1, &
\frac{2dy}{1+\rho^2+x^2+y^2}&\mapsto - i\sigma_2.
\end{aligned}
\end{equation}
Comparing \eqref{Clifford representation sphere} with \eqref{Clifford representation} and the identification $\cl_H(d\theta)=\cl_H(1)=I_2$, we see that
\begin{equation}
\label{R4 clifford isomorphisms}
\cl_S(\cdot) = \frac{2\rho}{1+\rho^2+x^2+y^2}  \exp(\tfrac{i\theta}{2}\sigma_3) \cl_H(\cdot) \exp(\tfrac{i\theta}{2}\sigma_3)
\end{equation}
in our chosen representations.  Let us compare this with the earlier expression \eqref{two clifford isomorphisms} relating $\cl_S$ and $\cl_H$.  The coefficient $2\rho/(1+\rho^2+x^2+y^2)$ is the conformal factor relating $g_S$ and $g_H$ and is equal to the coefficent $2|w|/(1+|w|^2)$ in \eqref{two clifford isomorphisms}.  The main differences between \eqref{R4 clifford isomorphisms} and \eqref{two clifford isomorphisms} are the two factors of $\exp(\frac{i\theta}{2}\sigma_3)$.  These appear because in \eqref{R4 clifford isomorphisms} we have chosen different frames for $S_+$, $S_-$ and $S$, whereas in \eqref{two clifford isomorphisms} we used a canonical identification $S_+\cong S\cong S_-$ over $S^4\setminus S^2$ to choose a common frame for all three bundles.  The Killing spinors $\psi_\alpha$ in \eqref{explicit Killing spinors} are written in the frame \eqref{Clifford representation} for $S$, and to write them in the frame \eqref{Clifford representation sphere} for $S_+$ we must multiply those expressions with $\exp(-\frac{i\theta}{2}\sigma_3)$.  So the expression for $\eqref{tilde psi}$ in our chosen frame for $S_+$ is
%By comparing this with \eqref{two clifford isomorphisms} we can see how our frames for $S$ and $S_\pm$ are related.  Thus the Killing spinors $\psi_\alpha$ written in the frame \eqref{Clifford representation} for $S$ correspond to $\exp(-\frac{i\theta}{2}\sigma_3)\psi_\alpha$ in the frame \eqref{Clifford representation sphere} for $S_+$.  Then from \eqref{tilde psi} and \eqref{explicit Killing spinors} we find
\begin{equation}
\label{spincartesianframe}
\tilde{\psi_1}=\left(\frac{4}{1+\rho^2+x^2+y^2}\right)^{\frac14} \begin{pmatrix}1\\0\end{pmatrix},\,
\tilde{\psi_2}=\left(\frac{4}{1+\rho^2+x^2+y^2}\right)^{\frac14} \begin{pmatrix}-x+iy\\\rho e^{ i\theta}\end{pmatrix}.
\end{equation}
These are smooth functions on $\R^4$ and they are written in a frame that extends over $\R^2\subset\R^4$, so we conclude that the spinors $\tilde{\psi}_\alpha$ extend smoothly over $\R^2\subset\R^4$.  
Repeating the calculation in a different coordinate patch shows that $\tilde{\psi_\alpha}$ extend smoothly over $S^2\subset S^4$.
\end{proof}

Proposition \ref{prop:b- extends} shows that there are natural maps
\begin{equation}\label{tangent space decomposition}
E_x^\pm\otimes K^\pm\to T_x^\C M^f_{n,p}.
\end{equation}
We now wish to show that:
\begin{theorem}\label{thm:decomposition}
The maps given in \eqref{tangent space decomposition} are isomorphisms.
\end{theorem}
\begin{proof}
We begin by considering surjectivity.  Let $Z_x^-$ denote the space of solutions $(a_-,\phi_-)$ of \eqref{linearised bog_eq}, \eqref{gauge fixing} such that $b_-=\phi_-\,d\theta+a_-$ extends smoothly to $S^4$.  Then \eqref{tangent space decomposition} is a composition of maps
\begin{equation}\label{EKZT}
E_x^-\otimes K^-\to Z_x^-\to T_x^\C M^f_{n,p}.
\end{equation}
Theorem \ref{thm:framed gauge fixing} shows that the map $Z_x^-\to T_x^\C M^f_{n,p}$ is surjective.  We now show that the map $E_x^-\otimes K^-\to Z_x^-$ is surjective.

Let $(a_-,\phi_-)\in Z_x^-$.  We wish to write this in the form
\begin{equation}\label{inverse map 1}
\cl_H(a_-,\phi_-)=\nu_1\otimes\psi_1^\ast+\nu_2\otimes\psi_2^\ast
\end{equation}
for some eigenspinors $\nu_i\in E_x^-$ and Killing spinors $\psi_i\in K^-$.  We begin by choosing a pair of linearly independent Killing spinors.  The Killing spinor equation implies that the function $(\psi_1,\psi_2)$ is constant; without loss of generality we can assume that this constant is 1 (it is certainly 1 for the choice made in \eqref{explicit Killing spinors}).  Then $\psi_2\otimes\psi_1^\ast-\psi_1\otimes\psi_2^\ast\in \Gamma(S\otimes S^\ast)$ corresponds to the identity map acting on $S$.  So \eqref{inverse map 1} holds with
\begin{equation}\label{inverse map 2}
\nu_\alpha=\epsilon_{\alpha\beta}\cl_H(a_-,\phi_-) \psi_\beta,
\end{equation}
where $\epsilon_{\alpha\beta}$ is totally antisymmetric and $\epsilon_{12}=1$.  By Proposition \ref{prop:Killing}, $\nu_\alpha$ solve the eigenspinor equation \eqref{eigenspinors}, so to confirm that they belong to $E_x^-$ we just need to check that $\tilde{\nu}_\alpha$ extend smoothly across $S^2\subset S^4$.

Let $\tilde{\psi}_\alpha$ be constructed from $\psi_\alpha$ as in \eqref{tilde psi}.  From equations \eqref{tilde nu}, \eqref{two clifford isomorphisms} and \eqref{tilde psi},
\begin{align}
\tilde{\nu}_\alpha
&= \epsilon_{\alpha\beta}\left(\frac{2|w|}{1+|w|^2}\right)^{-\frac32}e^{-\frac{i\theta}{2}}\frac{2|w|}{1+|w|^2}\cl_S(b_-)\left(\frac{2|w|}{1+|w|^2}\right)^{-\frac12}e^{-\frac{i\theta}{2}}\tilde{\psi}_\beta \\
&= \left( \frac{1+|w|^2}{w}\right) \epsilon_{\alpha\beta}\cl_S(b_-)\tilde{\psi}_\beta.\label{inverse map 3}
\end{align}
In order to show that $\tilde{\nu}_\alpha$ extends smoothly, we need to show that $\cl_S(b_-)\tilde{\psi}_\beta=O(w)$, in which the notation $O(w)$ means $w$ times a function that extends smoothly over $S^2\subset S^4$.  We do this calculation using the identification $S^1\times H^3\simeq S^4\setminus S^2$ with $\R^4\setminus\R^2$ that was introduced in the proof of Proposition \ref{prop:b- extends}.

Recall from Lemma \ref{lemma:BC b} that $b_-^{0,1}=O(w)$.  Turning to $b_-^{1,0}$, we deduce from $w/(1+|w|^2)=\rho e^{i\theta}/(1+\rho^2+x^2+y^2)$ that $dw=d(\rho e^{i\theta})+O(w)$.  So \eqref{Clifford representation sphere} gives
\begin{equation}
\cl_S:\frac{dw}{1+\rho^2+x^2+y^2}\mapsto\begin{pmatrix}O(w)&O(w)\\O(w)&i\end{pmatrix},\,
\frac{dz}{1+\rho^2+x^2+y^2}\mapsto\begin{pmatrix}0&i\\0&0\end{pmatrix}.
\end{equation}
From this and equations \eqref{tilde psi}, \eqref{explicit Killing spinors} we deduce that
\begin{equation}
\cl_S(b_-)=\cl_S(b_-^{1,0})+O(w)=\begin{pmatrix}O(w)&O(1)\\O(w)&O(1)\end{pmatrix},\quad
\tilde{\psi}_\beta = \begin{pmatrix}O(1)\\O(w)\end{pmatrix}.
\end{equation}
So $\cl_S(b_-)\tilde{\psi}_\beta$ equals a smooth section times $w$,  hence $\tilde{\nu}_\alpha$ given in \eqref{inverse map 3} extends smoothly across $\R^2\subset\R^4$.  Repeating the calculation in a different coordinate patch shows that $\tilde{\nu}_\alpha$ extends smoothly over $S^2\subset S^4$.
So $\nu_\alpha\in E_x^-$, and we have shown that \eqref{tangent space decomposition} is a surjective linear map.  Since the domain and target both have dimension $4n$, it is an isomorphism.
\end{proof}

%
%\begin{remark}\label{remark:uniqueness}
%The proof of Theorem \ref{thm:decomposition} also shows that the infinitesimal framed gauge transformation $\chi$ in Theorem \ref{thm:framed gauge fixing} is unique.  Uniqueness of the infinitesimal framed gauge transformation is equivalent to injectivity of the map $Z_x\to T_x^\C M_{n,p}^f$ in \eqref{EKZT}.  The proof of Theorem \ref{thm:decomposition} showed that the map $E_x^-\otimes K^-\to Z_x^-$ is surjective.  Since $\dim(E_x^-\otimes K^-)=4n$, this means that $\dim(Z_x^-)\leq 4n$.  The existence part of Theorem \ref{thm:decomposition} shows that the map $Z_x^-\to T_x^\C M_{n,p}^f$ is surjective.  Corollary \ref{cor:dimension Mf} shows that $\dim(T_x^\C M^f_{n,p})=4n\geq\dim(Z_x^-)$, so $Z_x\to T_x^\C M_{n,p}^f$ must be an isomorphism, and in particular injective.
%\qed
%\end{remark}
%

We are now in a position to prove the uniqueness part of Theorem \ref{thm:framed gauge fixing}:
\begin{proposition}\label{prop:uniqueness}
Let $(A,\Phi)$ and $(a,\phi)$ be as in Theorem \ref{thm:gauge fixing}.  Let $\chi^{(1)},\chi^{(2)}$ be two infinitesimal unframed $\mathfrak{sl}(2,\C)$ gauge transformations such that $(a_-^{(p)},\phi_-^{(p)})=(a+d^A\chi^{(p)},\phi+[\Phi,\chi^{(p)}])$ solve the gauge fixing condition \eqref{gauge fixing}.  Then $\chi^{(1)}=\chi^{(2)}$.
\end{proposition}
\begin{proof}
The statement of the proposition is equivalent to injectivity of the map $Z_x\to T_x^\C M_{n,p}^f$ in \eqref{EKZT}.  The proof of Theorem \ref{thm:decomposition} showed that the map $E_x^-\otimes K^-\to Z_x^-$ is surjective.  Since $\dim(E_x^-\otimes K^-)=4n$, this means that $\dim(Z_x^-)\leq 4n$.  Corollary \ref{cor:dimension Mf} shows that $\dim(T_x^\C M^f_{n,p})=4n$, so
\begin{equation}
\dim(Z_x^-)\leq\dim(T_x^\C M^f_{n,p}).
\end{equation}
The existence part of Theorem \ref{thm:framed gauge fixing} (which was proved in section \ref{sec:existence}) says that the map $Z_x^-\to T_x^\C M_{n,p}^f$ is surjective.  Since $\dim(Z_x^-)\leq\dim(T_x^\C M^f_{n,p})$, this map must also be injective, by the rank-nullity theorem.
\end{proof}

%A similar argument allows us to show that the framed gauge transformation that solves the gauge-fixing condition in Theorem \ref{thm:framed gauge fixing} is unique.  Suppose to the contrary: then there is a non-zero framed gauge transformation $\chi$ such that $(a_-,\phi_-)=(d^A\chi,0)$ solves \eqref{linearised bog_eq} and \eqref{gauge fixing}.  This can be written in the form \eqref{inverse map 1} with eigenspinors $\nu_\alpha\in E_x$ given in \eqref{inverse map 2}.  Since the map \eqref{tangent space decomposition} is injective and $(d^A\chi,0)$ represents the zero-vector in the quotient \eqref{tangent space definition}, it must be that $\nu_\alpha=0$.  Then, from \eqref{inverse map 1}, $d^A\chi=0$.  Since monopoles are irreducible this in turn implies that $\chi=0$, which contradicts our assumption.

\subsection{Symplectic and pluricomplex structures}
Now we discuss some some geometric structures that can be derived from the tensor decomposition \eqref{tangent space decomposition} of the tangent space.  Fix a non-zero $\psi \in K^+$, and let $\langle\psi\rangle\in\mathbb{P}K^+$ be the line through $\psi$.  We will first show that
\begin{proposition}\label{prop:complex structure}
There is a unique almost complex structure $J_{\langle\psi\rangle}$ on $TM^f_{n,p}$ whose holomorphic tangent bundle is
\begin{equation}
T^{1,0}_{\langle\psi\rangle}M^f_{n,p} = E^+\otimes\langle\psi\rangle.
\end{equation}
\end{proposition}
To prove this we will need to make use of the quaternionic structure on the spinor bundle.  The spinor bundle $S\to H^3$ admits a hermitian metric $\langle\cdot,\cdot\rangle$ and a quaternionic structure, i.e.\ an antilinear map $S\to S$, $s\mapsto\bar{s}$ that squares to $-1$.  These are related to the symplectic pairing by the identities $\langle s,t\rangle=(\bar{s},t)=\bar{s}^\ast(t)$.  In a suitable local frame, they can be written
\begin{equation}
s = \begin{pmatrix}s_1\\s_2\end{pmatrix}\mapsto \begin{pmatrix}\bar{s_2}\\ -\bar{s_1}\end{pmatrix} =: \bar{s},\quad \langle s,t\rangle = \bar{s}_1t_1+\bar{s}_2t_2.
\end{equation}
We extend complex conjugation to $\mathfrak{su } (2) $-valued spinors trivially, that is, for $ s \in S $, $ \upsilon \in \mathfrak{su }(2) $ we define $\overline{s \otimes \upsilon } = \overline{s }\otimes \upsilon $. Then
under the isomorphisms \eqref{clifford isomorphism}, \eqref{spinor isomorphism} the real structure on $(\Lambda^1\oplus\Lambda^0)\otimes\C$ given by complex conjugation corresponds to $\nu\otimes\psi ^\ast \mapsto\bar{\nu}\otimes\bar{\psi} ^\ast $.

Now the map $s\mapsto \bar{s}$ conjugates the Killing constant of any Killing spinor and conjugates the eigenvalue of any eigenspinor of $D^A-\ad(\Phi)$.  So $s\mapsto \bar{s}$ gives antilinear maps $E_x^\pm\to E_x^\mp$ and $K^\pm\to K^\mp$.  Therefore the two isomorphisms \eqref{tangent space decomposition} are exchanged by the real structure on $T^\C_xM^f_{n,p}$.

\begin{proof}[Proof of Proposition \ref{prop:complex structure}]
From the discussion above, $\nu\otimes\psi+\bar{\nu}\otimes\bar{\psi}$ is a real tangent vector for any $\nu \in E_x^+$.  The linear map $J_{\langle\psi\rangle}:T_xM^f_{n,p}\to T_xM^f_{n,p}$ will be given by
\begin{equation}
\label{J definition}
J_{\langle\psi\rangle}(\nu\otimes\psi+\bar{\nu}\otimes\bar\psi)=i\nu\otimes\psi-i\bar{\nu}\otimes\bar\psi\quad\forall\nu\in E_x^+.
\end{equation}
To ensure that this is well-defined we must check that every real tangent vector can be written in the form $\nu\otimes\psi+\bar{\nu}\otimes\bar\psi$, or in other words, that $\nu\mapsto \nu\otimes\psi+\bar{\nu}\otimes\bar\psi$ defines an isomorphism $E^-_x\to T_xM^f_{n,p}$.  This map is injective because if $\nu$ is in its kernel then $0=\nu\otimes\psi^\ast(\bar{\psi})+\bar{\nu}\otimes\bar\psi^\ast(\bar{\psi})=-\langle\psi,\psi\rangle\nu$, and $\langle\psi,\psi\rangle\neq0$ because $\psi\neq0$.  It is an isomorphism because its domain and target both have real dimension $4n$.
\end{proof}

Recall that an almost complex structure $J$ is compatible with a riemannian metric $g$ if and only if every holomorphic tangent vector $X$ is null, i.e.\ $g(X,X)=0$.  The almost complex structures $J_{\langle\psi\rangle}$ defined above are compatible with the bilinear pairing $g_+$ in a similar way:
\begin{proposition}\label{prop:g J}
Let $x\in M^f_{n,p}$ and let $\langle\psi\rangle\in K^\pm$.  Then for all $X\in T_xM^f_{n,p}$,
\begin{equation}
J_{\langle\psi\rangle}X = iX \implies g_+(X,X)=0.
\end{equation}
\end{proposition}
\begin{proof}
By definition, $J_{\langle\psi\rangle}X=iX$ if and only if $X=\nu\otimes\psi$ for some $\nu\in E^+_x$.  Under the isomorphism \eqref{clifford isomorphism}, the complex-linear inner product on $(\Lambda^1\oplus\Lambda^0)\otimes\C$ becomes $g(s\otimes t,s'\otimes t') = \frac12(s,s')(t,t')$.  So for any $\nu\otimes\psi\in E^+_x\otimes K^+$,
\begin{equation}
g_+(\nu\otimes\psi,\nu\otimes\psi)=-\frac{1}{4}\int_{H^3}(\psi,\psi)\tr\big(\nu,\nu\big)\mathrm{Vol}_{H^3}=0.
\end{equation}
\end{proof}

We have seen that the tensor product decomposition \eqref{tangent space decomposition} induces almost complex structures on the moduli space.  The decomposition also gives rise to tensor product decompositions of $g_\pm$, as we now explain.

Recall that $(\cdot,\cdot)$ denotes the symplectic form on the spinor bundle of $H^3$.  From this, we define $L^2$ skew bilinear forms $\omega_{E^\pm}$ on the spaces $E_x^\pm$ of eigenspinors:
\begin{equation}\label{omega definition}
\omega_{E^\pm}(\nu_1,\nu_2):=-\frac12\int_{H^3}\tr\big(\nu_1,\nu_2\big)\mathrm{Vol}_{H^3}.
\end{equation}
For the Killing spinors $K_\pm$, we have the following result:
\begin{lemma}
Let $\psi_1,\psi_2$ be two Killing spinors on $H^3$ with the same Killing constant.  Then $(\psi_1,\psi_2)$ is a constant function.
\end{lemma}
\begin{proof}
By the Killing spinor equation, for any tangent vector $X$:
\begin{align}
\nabla_X(\psi_1,\psi_2)&=(\nabla_X\psi_1,\psi_2)+(\psi_1,\nabla_X\psi_2)\\
&=\pm\frac{i}{2}\big[(X\cdot\psi_1,\psi_2)+(\psi_1,X\cdot\psi_2)\big]=0.
\end{align}
\end{proof}
For any $\psi_1,\psi_2\in K^\pm$ we define $\omega_{K^\pm}(\psi_1,\psi_2)\in\C$ to be the value of the constant function $(\psi_1,\psi_2)$ on $H^3$.  Then $\omega_{K^\pm}$ is a symplectic form on $K^\pm$.  We then have:
\begin{proposition}\label{prop:g omega}
The pairings $g_\pm,\omega_{E^\pm},\omega_{K^\pm}$ on $T^\C M^f_{n,p}$, $E^\pm$ and $K^\pm$ are related as follows:
\begin{equation}\label{omega g}
g_\pm = \frac12 \omega_{K^\pm}\otimes\omega_{E^\pm}.
\end{equation}
\end{proposition}
\begin{proof}
Under the isomorphism \eqref{clifford isomorphism}, the riemannian inner product on $TH^3$ corresponds to the symmetric pairing $\frac{1}{2}(\cdot,\cdot)\otimes(\cdot,\cdot)$ on $S\otimes S$.  So for any $\psi_\alpha\in K^\pm,\nu_\alpha\in E_x^\pm$,
\begin{align}
g_\pm(\nu_1\otimes\psi_1,\nu_2\otimes\psi_2)
&= -\frac14\int_{H^3}(\psi_1,\psi_2)\tr(\nu_1,\nu_2)\,\mathrm{Vol}_{H^3}\\
&= -\frac14(\psi_1,\psi_2)\int_{H^3}\tr(\nu_1,\nu_2)\,\mathrm{Vol}_{H^3}\\
&=\frac12\omega_{K^\pm}(\psi_1,\psi_2)\omega_{E^\pm}(\nu_1,\nu_2),
\end{align}
where we used the fact that $(\psi_1,\psi_2)$ is constant.
\end{proof}

The skew forms $\omega_{E^\pm}$ are reminiscent of a construction due to Nash \cite{nash:2007}.  To explain this, we first recall relevant part of \cite{nash:2007}.  Nash worked in minitwistor space, which is the space of oriented geodesics in $H^3$ and can  be identified with an open subset of $S^2_\infty\times S^2_\infty$.  From our perspective it is natural to identify this with $\mathbb{P}K^+\times\mathbb{P}K^-$.  A hyperbolic monopole corresponds to a spectral curve in minitwistor space such that a certain holomorphic bundle $L$ is trivial over this curve.  Nash defines a framed moduli space $M^t_{n,p}$ to be the space of spectral curves $\hat{S}$ in the total space of the bundle $L$ over minitwistor space.  He identifies his tangent space $T_x^\C M^t_{n,p}$ with the space of sections of the normal bundle $\hat{N}$ of this curve.  Nash obtains in eq.\ (28) of \cite{nash:2007} two decompositions
\begin{equation}\label{nash decomposition}
\begin{aligned}
T_x^\C M^t_{n,p} &\cong H^0(S,\hat{N}(-1,0))\otimes H^0(S,\mathcal{O}(1,0)) \\
T_x^\C M^t_{n,p} &\cong H^0(S,\hat{N}(0,-1))\otimes H^0(S,\mathcal{O}(0,1)).
\end{aligned}
\end{equation}
Here $\mathcal{O}(1,0)$ and $\mathcal{O}(0,1)$ are the duals of the tautological bundles of $\mathbb{P}K^\pm$; the spaces $H^0(S,\mathcal{O}(1,0))$, $H^0(S,\mathcal{O}(0,1))$ of sections are naturally identified with $K^\pm$.  So Nash's decompositions are very similar to \eqref{tangent space decomposition}.  In Lemma 3.5 Nash constructs skew-symmetric pairings on his spaces $H^0(S,\hat{N}(-1,0))$, $H^0(S,\hat{N}(0,-1))$ and shows that they are in fact symplectic forms.  We conjecture that his framed moduli space $M^t_{n,p}$ is naturally locally diffeomorphic to our $M^f_{n,p}$, that his decomposition \eqref{nash decomposition} coincides with \eqref{tangent space decomposition}, and that his symplectic forms coincide with \eqref{omega definition}.  If true, this would imply that $\omega_{E^\pm}$ are nondegenerate and hence, via \eqref{omega g}, that $g_\pm$ are nondegenerate.  It would also give a reduction of the structure group of $T^\C M^f_{n,p}$ to $Sp(2n,\C)$.

\begin{remark}
From $\omega_{E^\pm}$, one can construct a family of holomorphic (symplectic) forms on the tangent bundle.  These are parametrised by $\psi\in K^\pm$ and defined by
\begin{equation}
\Omega_{\psi}^\pm((a_1,\phi_1),(a_2,\phi_2)):= \omega_{E^\pm}((a_1+\phi_1)\cdot\psi,(a_2+\phi_2)\cdot\psi).
\end{equation}
With respect to the complex structures $J_{\langle\psi\rangle}$, $\Omega_{\psi}^-$ is of type $(2,0)$ and $\Omega_{\psi}^+$ is of type $(0,2)$.
\end{remark}

\section{Calculating the metric}
\label{sec:calculating}
In this section we obtain some explicit solutions of the gauge-fixing condition \eqref{gauge fixing} and hence evaluate some components of the bilinear form $g_\pm$.  In particular, we calculate $g_\pm$ for the 1-monopole moduli space and confirm that it is a riemannian metric.

Let $(A,\Phi)$ be a monopole of any charge and any mass parameter.  Then
\begin{equation}\label{circle tangent vector}
(a,\phi)_0:=(d^A\Phi,0)
\end{equation}
solves the linearised Bogomolny equation \eqref{linearised bog_eq} and both gauge-fixing conditions \eqref{gauge fixing} because $(A,\Phi)$ solves the Bogomolny equation \eqref{bog_eq} and hence $d^A\ast d^A\Phi = d^AF^A=0$.  The solution \eqref{circle tangent vector} is induced by the unframed gauge transformation $\Phi$.  However, $\Phi$ is not a framed gauge transformation, because $\int_{H^3}\tr(F^A\wedge d^A\Phi)=-4\pi np\neq0$.  So \eqref{circle tangent vector} is a non-zero tangent vector to $M^f_{n,p}$.  Inserting \eqref{circle tangent vector} into \eqref{g+-} gives
\begin{equation}
g_\pm((a,\phi)_0,(a,\phi)_0)=-\frac{1}{2}\int_{H^3}\tr(d^A\Phi\wedge\ast d^A\Phi)=E = 2\pi np.
\end{equation}

From \eqref{circle tangent vector} and two independent Killing spinors $\psi_1,\psi_2\in K^\pm$ we can construct two solutions
\begin{equation}\label{F eigenspinors definition}
\nu_\alpha=- d^A\Phi\cdot\psi_\alpha=F^A\cdot\psi_\alpha
\end{equation}
of the Dirac eigenspinor equation \eqref{eigenspinors}.  Proposition \ref{prop:Killing} shows that these solve \eqref{eigenspinors}, and moreover that $\nu_\alpha\otimes\psi_\beta^\ast$ give four solutions of the linearised Bogomolny equation \eqref{linearised bog_eq} and gauge-fixing condition \eqref{gauge fixing}.  If $\psi_1,\psi_2$ are chosen such that $(\psi_1,\psi_2)=1$ as in \eqref{explicit Killing spinors} then $\psi_2\otimes\psi_1^\ast-\psi_1\otimes\psi_2^\ast$ is the identity map on $S$ and so
\begin{equation}\label{infingaugetransf}
\nu_1\otimes\psi_2^\ast-\nu_2\otimes\psi_1^\ast = d^A\Phi
\end{equation}
corresponds to the tangent vector \eqref{circle tangent vector} that we started with.  In this equation, and throughout this section, we implicitly make use of the isomorphism \eqref{clifford isomorphism} between $\Lambda^1\oplus\Lambda^0$ and $S\otimes S^\ast$, and omit the notation $\cl_H$.

The remaining three tangent vectors can be interpreted using isometries.  Given any Killing vector field $Z$, $(a,\phi)_Z:=(Z\lrcorner F^A,Z\lrcorner d^A\Phi)$ differs from the Lie derivative $(L_Z A,L_Z\Phi)$ of the monopole $(A,\Phi)$ by a gauge transformation.  So it describes a tangent vector to the moduli space induced by $Z$ and solves the linearised Bogomolny equation \eqref{linearised bog_eq}, but does not necessarily solve the gauge-fixing condition \eqref{gauge fixing}.

The isometry group of hyperbolic space is isomorphic to the Lorentz group $SO (3,1)$.  This is generated by Killing vector fields $X _i , Y _j$ satisfying
 \begin{equation}
[ X _i , X _j ] =- \epsilon _{ ijk }X _k , \quad [ X _i , Y _j ] =- \epsilon _{ ijk }Y _k , \quad [ Y _i , Y _j ] =\epsilon _{ ijk }X _k .
\end{equation} 
In upper half-space coordinates they are given by
\begin{equation}
\label{kvfh3upper} 
\begin{aligned}
X _1 &=  y  Y _3  +  \frac{1}{2} ( 1 - |\mathbf{x} |  ^2  )\partial _y ,&
Y _1 &=  - x Y _3   + \frac{1}{2} (1 + |\mathbf{x} | ^2  ) \partial _x  ,\\
X _2 &=  -xY _3   - \frac{1}{2} (1 - |\mathbf{x} |  ^2  )\partial _x,&
Y _2 &=  - y  Y _3   + \frac{1}{2} (1  + |\mathbf{x} | ^2  ) \partial _y ,\\
X _3 &=  x  \partial _y  - y \partial _x,&
Y _3 &=  x  \partial _x + y \partial _y + \rho  \partial _ \rho ,
%X _1 &=  y  Y _3  +  \frac{1}{2} ( 1 - |x|  ^2  )\partial _y  ,\\
%X _2 &=  -xY _3   - \frac{1}{2} (1 - |x|  ^2  )\partial _x  ,\\
%X _3 &=  x  \partial _y  - y \partial _x,\\
%Y _1 &=  - x Y _3   + \frac{1}{2} (1 + |x| ^2  ) \partial _x  ,\\
%Y _2 &=  - y  Y _3   + \frac{1}{2} (1  + |x| ^2  ) \partial _y ,\\
%Y _3 &=  x  \partial _x + y \partial _y + \rho  \partial _ \rho ,
\end{aligned}
\end{equation} 
where $|\mathbf{x} | ^2 =x ^2 + y ^2 + \rho ^2 $.  We compare these with the vector fields $Z_{\psi_\alpha,\psi_\beta}$ given by
\begin{equation}
Z_{\psi_\alpha,\psi_\beta} =-\psi_\alpha\otimes\psi_\beta^\ast-\psi_\beta\otimes\psi_\alpha^\ast
\end{equation}
under the isomorphism \eqref{clifford isomorphism}.  For concreteness, we now choose the lower sign in the gauge-fixing condition \eqref{gauge fixing} and choose $\psi_1,\psi_2\in K^-$ as in \eqref{explicit Killing spinors}.  Using \eqref{Clifford representation} one calculates e.g.\
\begin{equation}
Z_{\psi_1,\psi_2}  = \begin{pmatrix}
1 & 2 (x-iy)\rho ^{-1}  \\
0 & -1
\end{pmatrix}  =-X _3 +i Y _3 .
\end{equation} 
Proceeding similarly we find
\begin{equation}
\begin{split} 
Z _1 :=Y_1+iX _1  &= i(\psi_2\otimes\psi_2^\ast-\psi_1\otimes\psi_1^\ast),\\
Z _2 := Y_2+iX _2 &=-(\psi_1\otimes\psi_1^\ast+\psi_2\otimes\psi_2^\ast), \\
Z _3 :=Y_3+iX _3  &=i(\psi_1\otimes\psi_2^\ast+\psi_2\otimes\psi_1^\ast) .
 \end{split} 
\end{equation}
In general, the Clifford multiplication of the two-form $F^A$ with a vector field $Z$ is given by
\begin{equation}
F^A\cdot Z = F^A\wedge Z - F^A\llcorner Z = (Z\lrcorner\ast F^A)\mathrm{Vol}+Z\lrcorner F^A = Z\lrcorner d^A\Phi + Z\lrcorner F^A,
\end{equation}
in which we used that $\mathrm{Vol}$ acts as identity in the Clifford algebra.  Thus
\begin{equation} 
\label{symmtransf} 
\begin{split} 
i(\nu _2 \otimes \psi _2 ^\ast - \nu _1 \otimes \psi _1 ^\ast)&=(Z_1\lrcorner F^A,Z_1\lrcorner d^A\Phi)=:(a,\phi)_1    ,\\
-(\nu _1 \otimes \psi _1 ^\ast  + \nu _2 \otimes \psi _2 ^\ast) &=(Z_2\lrcorner F^A,Z_2\lrcorner d^A\Phi)=:(a,\phi)_2,\\
i( \nu _1 \otimes \psi _2 ^\ast + \nu _2 \otimes \psi _1 ^\ast )&=(Z_3\lrcorner F^A,Z_3\lrcorner d^A\Phi)=:(a,\phi)_3
\end{split} 
\end{equation}
are solutions of the linearised Bogomolny equation \eqref{linearised bog_eq} and the gauge fixing condition \eqref{gauge fixing}.  They are complexified tangent vectors and can be interpreted as complex linear combinations of boosts and rotations of the monopole.

We have now identified four complexified tangent vectors \eqref{infingaugetransf}, \eqref{symmtransf}, written in terms of $\nu_\alpha\in E_x^-$ and $\psi_\alpha\in K^-$.  We can evaluate the bilinear form $g_-$ on them with the help of \eqref{omega g}, which relates these to skew-symmetric pairings on $K^-$, $E_x^-$.  Working in upper half space coordinates with the Clifford algebra representation (\ref{Clifford representation}) we get
\begin{equation}
F ^A  =-i \rho ^2 (F _{ xy } \sigma _3 + F _{ y \rho }\sigma _1 + F _{ \rho x } \sigma _2 ) = -i \rho ^2 \begin{pmatrix}
F _{ x y } & F _{ y \rho } - i F _{ \rho x }  \\
F _{ y \rho }+ i F _{ \rho x } & - F _{ xy }
\end{pmatrix} .
\end{equation}
Then from \eqref{explicit Killing spinors} and \eqref{F eigenspinors definition} we compute
\begin{equation}
\label{F eigenspinors} 
\begin{split} 
\nu _1 & = - i \rho ^{ 3/2 } \begin{pmatrix}
 F _{ x y }  \\
 F _{ y \rho }+ i F _{ \rho x }
\end{pmatrix} , \\
\nu _2 & =  i \rho ^{ 3/2 } \begin{pmatrix}
(x - i y )F _{ xy } - \rho (F _{ y \rho }- i F _{ \rho x } )\\
(x - i y ) (F _{ y \rho }+ i F _{ \rho x } ) +  \rho F _{ xy }
\end{pmatrix} .
\end{split} 
\end{equation}
Thus
\begin{equation}
\begin{split} 
\omega_{E^-}(\nu_1,\nu_2) &
= - \frac{1}{2} \int _{ H ^3 } \operatorname{Tr} (\nu _1 , \nu _2 ) \operatorname{Vol} _{ H ^3 }  \\ &
= - \frac{1}{2} \int _{ H ^3 } \rho ^4  \operatorname{Tr} ( F _{ xy } ^2 + F _{ y \rho } ^2 + F _{ \rho x }^2 ) \operatorname{Vol} _{ H ^3 } \\ &
= - \frac{1}{2} \int _{ H ^3 } \operatorname{Tr} (F  ^A \wedge *  F ^A  )
= 2\pi np,\label{main omega int}
\end{split} 
\end{equation} 
using (\ref{energy functional}).  Thus, from \eqref{omega g},
\begin{equation} 
\label{main en int} 
g_-( \nu_\alpha \otimes \psi _\beta ^\ast , \nu _\gamma \otimes \psi _\delta ^\ast  ) =\pi n p\,\epsilon_{\alpha\gamma}\epsilon_{\beta\delta}
\end{equation} 
It follows that
\begin{equation}
\label{monopole inertia tensor} 
g_-( (a , \phi ) _{\mu}, (a , \phi )_{\nu})  = 2\pi n p\,  \delta _{ \mu\nu },\quad \text{for}\,\mu,\nu\in\{0,1,2,3\}.
\end{equation} 

Equation \eqref{monopole inertia tensor} is a satisfyingly simple formula for the symmetry-induced components of $g_-$.  It is also consistent with an interpretation of charge $n$ monopoles as ensembles of $n$ point particles, each with mass $2\pi p$.  To see this, we note that the Killing vector fields \eqref{kvfh3upper} satisfy
\begin{equation}
g(Z_j,Z_k)=g(Y_j,Y_k)-g(X_j,X_k)+ig(X_j,Y_k)+ig(Y_j,X_k)=\delta_{jk}
\end{equation}
with respect to the metric $g$ on $H^3$.  The motion of a particle of mass $m$ in $H^3$ is described by the metric $mg$.  The motion of $n$ particles of equal mass is described by the product metric $g_n=mg+\ldots+mg$ on $(H^3)^n$.  So we find that
\begin{equation}
g_n(Z_j,Z_k)=g_n(Y_j+iX_j,Y_k+iX_k)=nm\delta_{jk},
\end{equation}
in which $X_j,Y_j$ represent simultaneous boosts and rotations of all $n$ particles.  This agrees with our result \eqref{monopole inertia tensor} for the choice $m=2\pi p$.  The total mass $2\pi np$ is equal to the energy $E$ of the monopole.  It would be interesting to also compare inner products of real Killing vectors $X_i,Y_j$ for the point particle and monopole metrics.  However, we are not able to do this because we are currently only able to solve the gauge fixing condition \eqref{gauge fixing} for the complex combinations $Y_j+iX_j$.  For point particles, the inner products $g(X_j,X_k)$ etc depend on the position of the particle in $H^3$, so one would expect that for monopoles $g_-(X_j,X_k)$ depends on the moduli space coordinates, unlike the inner products \eqref{monopole inertia tensor}.

Equation \eqref{monopole inertia tensor} fully determines $g_-$ in the case of charge 1.  It is well-known that the Bogmolny equations admit a spherically-symmetric solution when $n=1$.  The moduli space $M^f_{1,p}$ can be generated by transforming this monopole using isometries of $H^3$ and unframed gauge transformations $g=\exp(s\Phi)$.  It is therefore diffeomorphic to $\R\times SO(3,1)/SO(3)=\R\times H^3$.  The bilinear form $g_-$ is clearly invariant under the transitive action of $SO(3,1)\times\R$, so it is determined by its value at a single point in the moduli space.

Consider the point in the moduli space corresponding to the spherically-symmetric 1-monopole.  This is identified with the point $(x,y,\rho)=(0,0,1)$ in $H^3$.  Rotation of this monopole is equivalent to a gauge transformation, so
\begin{equation}
(X_j\lrcorner F^A,X_j\lrcorner d^A\Phi)=(d^A\chi_j,[\Phi,\chi_j])
\end{equation}
for some gauge transformations $\chi_j$.  These gauge transformations are framed,
because
\begin{multline}
\int_{H^3}\tr(d^A\chi_j\wedge\ast d^A\Phi)
=\int_{H^3}\tr(X_j\lrcorner F^A\wedge F^A)\\
=\frac12\int_{H^3}X_j\lrcorner\tr(F^A\wedge F^A)
=\int_{H^3}0
=0.
\end{multline}
So
\begin{equation}
\label{rotgeqzero1monopole} 
(a,\phi)_j=(Z_j\lrcorner F^A,Z_j\lrcorner d^A\Phi)=(Y_j\lrcorner F^A+id^A\chi_j,Y_j\lrcorner d^A\Phi+i[\Phi,\chi_j])
\end{equation}
and the tangent vectors $(a,\phi)_j$ represent boosts of the monopole.  They can therefore be identified with the tangent vectors $Y_j$ given in \eqref{kvfh3upper}, evaluated at the point $(0,0,1)\in H^3$.  By direct calculation, $g_{H^3}(Y_j,Y_k)=\delta_{jk}$ at this point.  Comparing with \eqref{monopole inertia tensor}, we conclude that the metric on $M^f_{1,p}$ is
\begin{equation}
\label{1mon metric}
g_- = 2\pi p(g_\R+g_{H^3}).
\end{equation}
Since this is real it agrees with $g_+$.

In the case $p=\frac12$ the result \eqref{1mon metric} for the metric on $M^f_{1,\frac12}$ can also be verified by a direct calculation based on explicit formulae for the monopole.  We include these calculations here as they provide a consistency check on our results, and they might be useful in future calculations for $n>1$.  A monopole in $M^f_{1,\frac12}$ corresponds to a circle-invariant instanton with $c_2=2np=1$.  By writing the instanton using the 't Hooft ansatz and changing to a circle-invariant gauge we obtain:
\begin{equation} 
\label{1-monopoleA} 
\begin{split} 
A = \frac{i}{\lambda ^2 + |\mathbf{x} - \mathbf{x} _0 | ^2 } \Big[&( \rho \sigma _2 - (y- y _0 ) \sigma _3 ) d  x + ( (x- x _0 ) \sigma _3  -\rho \sigma _1) d  y \\  &+( (y- y _0 ) \sigma _1 - (x- x _0 ) \sigma _2 )d  \rho \Big] ,
\end{split} 
\end{equation}
\begin{equation} 
 \label{1-monopolePhi} 
\Phi = \frac{i}{2} \sigma _3  - \frac{ i\rho }{\lambda ^2 + |\mathbf{x} - \mathbf{x} _0 | ^2 } (\rho \sigma _3 + (x- x _0 ) \sigma _1 + (y- y _0 ) \sigma _2),
\end{equation}
in which $x_0,y_0,\lambda$ parametrise the monopole centre with respect to upper half space coordinates on $H ^3 $ and $|\mathbf{x}-\mathbf{x}_0|^2=(x-x_0)^2+(y-y_0)^2+\rho^2$.  The corresponding field strength is 
\begin{equation}
F _{ xy } = i f \sigma _3 , \quad F _{ y \rho } =i f \sigma _1 , \quad F _{ \rho x } =i f \sigma _2 ,
\end{equation} 
with
\begin{equation}
\label{harmonic1-mon} 
f =\frac{2 \lambda ^2 }{ ( |\mathbf{x} - \mathbf{x} _0 |^2  + \lambda ^2 )^2 }.
\end{equation} 
Using (\ref{F eigenspinors}) one calculates 
\begin{equation}
\operatorname{Tr} (\nu _1 , \nu _2 ) =  - 6\rho ^4 f ^2 .
\end{equation}
Therefore
\begin{multline}
\label{explicitmetriccomp} 
- \frac{1}{2} \int _{ H ^3 } \operatorname{Tr} (\nu _1 , \nu _2 )\operatorname{Vol} _{ H ^3 } 
= \\
\int _{ H ^3 } \frac{  12 \lambda ^4 \rho ^4  }{( (x- x _0 ) ^2 +  (y- y _0 ) ^2 + \rho ^2 + \lambda ^2  )^4 } \frac{d  x \, d  y \, d  \rho }{\rho ^3 }
= \pi ,
\end{multline}
which agrees with (\ref{main omega int}) since $n =1 $ and $p =\tfrac{1}{2} $.  The metric $g_-$ is therefore given by \eqref{main en int}.

We can also obtain (\ref{rotgeqzero1monopole}) by a direct calculation.  To do so, we must choose $\lambda=1$, $x_0=y_0=0$ and find functions $\chi _j $  such  that 
\begin{equation}%\label{Yjchij}
(a,\phi)=
(Y_j\lrcorner F^A,Y_j\lrcorner d^A\Phi) +i( d  _A \chi  _j , [ \Phi , \chi _j ])
\end{equation}
is equal to a linear combination of $\nu_\alpha\otimes\psi_\beta$.
For example, take $j =1 $. One has
\begin{align}
\label{Y1monopoleAvariation} 
Y _1\lrcorner F^A &= \frac{i  }{(1 + |\mathbf{x}  |^2 )^2 } 
\Big[
(2x \rho \sigma _1  + (1 - x ^2  + y ^2 + \rho ^2 ) \sigma _3 )d  y
\nonumber\\
&
- (2x y \sigma _1  +  (1 - x ^2  + y ^2 + \rho ^2 ) \sigma _2 )d  \rho 
+ 2 x   ( y \sigma _3 -\rho \sigma _2 ) d  x
\Big] ,\\
\label{Y1monopoleHiggsvariation} 
Y _1\lrcorner d^A\Phi
&=\frac{i \rho  }{(1 + |\mathbf{x}  |^2 )^2 } [ 2x (y \sigma _2  + \rho \sigma _3 )-(1 - x ^2  + y ^2 + \rho ^2 ) \sigma _1 ].
\end{align} 
It is now useful to note that
\begin{equation}
\operatorname{Tr}  \left[   \begin{pmatrix}
\Phi  & 0 \\
0 & \Phi 
\end{pmatrix}(\phi + a )   \right] 
=2 \operatorname{Tr} ( \Phi \phi )
\end{equation} 
is a gauge invariant quantity. For $(a,\phi)=(Y_1\lrcorner F^A,Y_1\lrcorner d^A\Phi)$ we find using (\ref{1-monopolePhi}) and (\ref{Y1monopoleHiggsvariation}) that
\begin{equation}
2 \operatorname{Tr} ( \Phi \phi ) 
= - \frac{8 x \rho ^2 }{ (1 + |\mathbf{x} | ^2 ) ^3 }.
\end{equation} 
We also have
\begin{align} 
G _1 :=2 \operatorname{Tr} \left[  
\nu _1 \otimes \psi _1 ^\ast 
\begin{pmatrix}
\Phi  & 0 \\
0 & \Phi 
\end{pmatrix}\right]   &=-\frac{4i\rho ^2  (x +  i y )}{ (1 + |\mathbf{x} |^2 )^3 },\\
G _2 :=2 \operatorname{Tr} \left[  
\nu _1 \otimes \psi _2 ^\ast 
\begin{pmatrix}
\Phi  & 0 \\
0 & \Phi 
\end{pmatrix}\right]   &
=2 \operatorname{Tr} \left[  
\nu _2 \otimes \psi _1 ^\ast 
\begin{pmatrix}
\Phi  & 0 \\
0 & \Phi 
\end{pmatrix}\right] 
=\frac{2i\rho ^2  (|\mathbf{x} |^2-1  )}{ (1 + |\mathbf{x} |^2 )^3 },\\
G _3 :=2 \operatorname{Tr} \left[  
\nu _2 \otimes \psi _2 ^\ast 
\begin{pmatrix}
\Phi  & 0 \\
0 & \Phi 
\end{pmatrix}\right]   &= \frac{4i\rho ^2  (x -  i y )}{ (1 + |\mathbf{x} |^2 )^3 }.
\end{align} 
Since
\begin{equation}
i (G _3 - G _1) =2 \operatorname{Tr} ( \Phi \phi ) ,
\end{equation}
it must be that \eqref{Y1monopoleAvariation}, \eqref{Y1monopoleHiggsvariation} corresponds to $i\nu_2\otimes\psi_2^\ast-i\nu_1\otimes\psi_1^\ast$, in agreement with \eqref{symmtransf}.  We therefore look for a function $\chi _1  $ satisfying
\begin{equation}
\label{Y1conditions} 
(Y_1\lrcorner F^A , Y_1\lrcorner d^A\Phi )  + i( d  _A \chi_1  ,  [ \Phi , \chi _1 ] ) =  
i ( \nu _2 \otimes \psi _2 ^\ast - \nu _1 \otimes \psi _1 ^\ast  ).
\end{equation} 
A computation shows that (\ref{Y1conditions}) is equivalent to the conditions
\begin{align}
 i[ \Phi , \chi _1 ] 
&= \frac{\rho (2 \rho y \sigma _3 + 2 xy \sigma _1 + (1- x ^2 + y ^2  - \rho ^2 ) \sigma _2 ) }{ (1 + |\mathbf{x} |^2 )^2}  ,\\
id ^A \chi _1 
&= \frac{1}{(1 + | \mathbf{x}  |^2 )^2 } \big[
d  \rho  \left(\sigma _1 \left(\rho   ^2+x^2-y^2-1\right)+2x   y\sigma _2\right)
\nonumber\\
&
+2y d  y  \left(\rho  \sigma _1 - x\sigma _3   \right)
-d  x \left(\sigma _3 \left(\rho   ^2+x^2-y^2-1\right)+ 2 \rho y  \sigma _2   \right)
\big]. 
\end{align}
The above equations have the unique solution
\begin{equation}
\chi _1 =  \frac{i (\rho \sigma _1 -x \sigma _3) }{1 + |\mathbf{x} |^2 }.
\end{equation} 
With similar computations we find that 
\begin{align}
(Y_2\lrcorner F^A,Y_2\lrcorner d^A\Phi)   + i( d  _A \chi _2 , [ \Phi , \chi _2 ])   &
=  -(\nu _1 \otimes \psi _1 ^\ast  + \nu _2 \otimes \psi _2 ^\ast)  ,\\
(Y_3\lrcorner F^A,Y_3\lrcorner d^A\Phi)   + i( d  _A \chi _3 , [ \Phi , \chi _3 ])   &
= i( \nu _1 \otimes \psi _2 ^\ast + \nu _2 \otimes \psi _1 ^\ast )  , 
\end{align}
with
\begin{equation}
\chi _2 =\frac{i (\rho \sigma _2-y \sigma _3) }{1 + |\mathbf{x} |^2 },
\quad \chi _3 = \frac{ i\sigma _3  }{1 + |\mathbf{x} |^2 } + \Phi ,
\end{equation}
which confirms the result \eqref{symmtransf}, (\ref{rotgeqzero1monopole}).

%\medskip
%\noindent\textbf{Declarations}

\medskip
\noindent\textbf{Funding} Partial financial support was received from a London Mathematical Society Scheme 4 Research in Pairs grant (ref.~42324).

%\medskip
%\noindent\textbf{Competing interests} The authors have no competing interests to declare that are relevant to the content of this article.

%\medskip
%\noindent\textbf{Data availability} No datasets were generated or analysed during the current study.

\appendix
\section{Alternative framings and complex structures}
\label{appendix}

%We wish to show that the almost complex structures $J_{\langle\psi\rangle}$ are integrable.  We will do so by showing that they coincide with almost complex structures on the moduli space of framed instantons, which are known to be integrable.  To make this link, we need to consider alternative framings of the moduli space.

In this appendix we compare our definition of the framed moduli space with definitions that have been used by others.  We also compare the almost complex structures $J_{\langle\psi\rangle}$ with some known integrable complex structures.

The traditional definition of the framed moduli space is as follows \cite{atiyah:1987}.  Let us choose a point in $S^2_\infty$, denoted by $\infty\in S^2_\infty$.  The group of gauge transformations framed at $\infty$ is the group $G^\infty$ of unframed gauge transformations $g:H^3\to SU(2)$ such that $g(\infty)=\mathrm{id}_V$.  The \emph{moduli space of hyperbolic monopoles framed at $\infty$} is
\begin{equation}
M^\infty_{n,p}=C_{n,p}/G^\infty.
\end{equation}

To compare $M^\infty_{n,p}$ with $M^f_{n,p}$, we introduce a third moduli space.  Let $T$ be the section of $\mathrm{End}(V)|_{S^2_\infty}$ whose eigenspaces $L^\pm$ have eigenvalues $\pm1$.  Let $(A^o,\Phi^o)\in C_{n,p}$ be a fixed monopole.  The monopole boundary conditions imply that, for any other monopole $(A,\Phi)\in C_{n,p}$, there exists a real 1-form $\alpha$ on $S^2_\infty$ such that $(A-A^o)|_{S^2_\infty}=i\alpha T$.   A monopole $(A,\Phi)\in C_{n,p}$ is said to be \emph{framed at the boundary} if this 1-form satisfies $d\ast_{S^2_\infty} \alpha=0$.  The space of monopoles framed at the boundary is denoted by $C^b_{n,p}$.  Let $G^b\subset G^\infty$ be the group of unframed gauge transformations satisfying $g|_{S^2_\infty}=\mathrm{id}_V$.  This group clearly acts on $C^b_{n,p}$, and we define the \emph{moduli space of boundary-framed hyperbolic monopoles} to be
\begin{equation}
M^b_{n,p}=C^b_{n,p}/G^b.
\end{equation}
%Finally, for fixed $(A,\Phi)\in C_{n,p}$ we denote by $G^f_{A,\Phi}$ the group of gauge transformations $g\in G^u$ such that $CS[A,g\cdot A]=0$.

The moduli spaces $M^\infty_{n,p}$ and $M^b_{n,p}$ admit actions of $s\in\R$ given by gauge transformation with $g_s=\exp(s\Phi)$.  The eigenvalues of $\Phi$ on $S^2_\infty$ are $\pm ip$, so when $s= \pi/p$, $-g_s\in G^b\subset G^\infty$.  The action of $-g_s$ on a monopole is the same as the action of $g_s$, so $s=\pi/p$ acts trivially on $M^\infty_{n,p}$ and $M^b_{n,p}$.  Therefore the actions of $\R$ descend to free actions of $S^1=\R/\frac{\pi}{p}\mathbb{Z}$.
\begin{proposition}\label{prop:3 framings}
There are natural diffeomorphisms $M^b_{n,p}\to M^\infty_{n,p}$ and $M^b_{n,p}/\mathbb{Z}_n\to M^f_{n,p}/8\pi^2\mathbb{Z}$.
\end{proposition}
\begin{proof}
Since $C^b_{n,p}\subset C_{n,p}$ and $G^b\subset G^\infty$ there is a natural map
\begin{equation}
M^b_{n,p}=C^b_{n,p}/G^b\to C_{n,p}/G^\infty=M^\infty_{n,p}.
\end{equation}
We will show that this is a bijection and that it is a diffeomeorphism.

For surjectivity, suppose that $(A,\Phi)\in C_{n,p}$ represents a point in $M^\infty_{n,p}$.  We must find a $g\in G^\infty$ such that $g\cdot(A,\Phi)\in C^b_{n,p}$.  The boundary conditions for hyperbolic monopoles imply that $(A-A^o)|_{S^2_\infty}=i\alpha T$ for some real 1-form $\alpha$.  By solving the Poisson equation $d\ast_{S^2_\infty} d\lambda = d\ast_{S^2_\infty}\alpha$, we can find a $\lambda:S^2_\infty\to\R$ such that $\alpha-d\lambda$ is coclosed.  By subtracting a constant, we can arrange that $\lambda(\infty)=0$.  We can then choose a gauge transformation $g\in G^\infty$ such that $g|_{S^2_\infty}=\exp(i\lambda T)$ (there is no topological obstruction to doing so because $H_2(SU(2),\mathbb{Z})=0$).  The gauge-transformed monopole $g\cdot(A,\Phi)=(gAg^{-1}-dgg^{-1},g\Phi g^{-1})$ will satisfy $(g\cdot A-A^o)|_{S^2_\infty}=i(\alpha-d\lambda)T$ so will belong to $C^b_{n,p}$.

For injectivity, suppose that $(A^1,\Phi^1),(A^2,\Phi^2)\in C^b_{n,p}$ map to the same point in $M^\infty_{n,p}$.  Then there exists $g\in G^\infty$ such that $(A^2,\Phi^2)=g\cdot (A^1,\Phi^1)$.  We will show that $g\in G^b$ so that $(A^1,\Phi^1)$ and $(A^2,\Phi^2)$ determine the same point in $M^b_{n,p}$.  The gauge transformation $g$ is such that $g|_{S^2_\infty}=\exp(i\lambda T)$ for some function $\lambda:S^2_\infty\to \R$ satisfying $\lambda(\infty)=0$.  The connections $A^1,A^2$ satisfy $(A^i-A^o)|_{S^2_\infty}=i\alpha^iT$ for some coclosed 1-forms $\alpha^1,\alpha^2$.  Since $g A^1 g^{-1}-dg g^{-1}=A^2$, $\alpha^1-\alpha^2=d\lambda$.  Since $d\ast_{S^2_\infty}\alpha^i=0$, the function $\lambda$ is harmonic and hence constant.  Since $\lambda(\infty)=0$, $\lambda=0$ on the whole of $S^2_\infty$.  Therefore $g\in G^b$ as claimed.

To show that the map is a diffeomorphism, we consider the induced map on tangent vectors.  Fix $(A,\Phi)\in C^b_{n,p}$ representing a point in $M^b_{n,p}$.  Then 
\begin{equation}
T_{A,\Phi}M^b_{n,p}=T_{A,\Phi}C^b_{n,p}/\g^b_{A,\Phi}\quad\text{and}\quad
T_{A,\Phi}M^\infty_{n,p}=T_{A,\Phi}C_{n,p}/\g^\infty_{A,\Phi},
\end{equation}
where
\begin{align}
T_{A,\Phi}C^b_{n,p}&=\big\{(a,\phi)\in T_{A,\Phi}C_{n,p}\::\:a|_{S^2_\infty}=i\alpha T\text{ and }d\ast_{S^2_\infty}\alpha=0\big\},\\
\g^b_{A,\Phi}&=\big\{(d^A\chi,[\Phi,\chi])\in \g^u_{A,\Phi}\::\:\chi|_{S^2_\infty}=0\big\},\label{boundary IGT}\\
\g^\infty_{A,\Phi}&=\big\{(d^A\chi,[\Phi,\chi])\in \g^u_{A,\Phi}\::\:\chi(\infty)=0\big\},
\end{align}
and $\g^u_{A,\Phi}$ was defined on page \pageref{gu definition}.  We must show that the natural map $T_{A,\Phi}M^b_{n,p}\to T_{A,\Phi}M^\infty_{n,p}$ is a bijection.

For surjectivity, let $(a,\phi)\in T_{A,\Phi}C_{n,p}$.  We must find a $\chi$ such that $(a-d^A\chi,\phi-[\Phi,\chi])\in T_{A,\Phi}C^b_{n,p}$ and $\chi(\infty)=0$.  By writing $\chi|_{S^2_\infty}=i\lambda T$, this requirement is equivalent to solving the Poisson equation $d\ast_{S^2_\infty}d\lambda = d\ast_{S^2_\infty}\alpha$ and $\lambda(\infty)=0$ on $S^2_\infty$.  The Poisson equation has a unique solution satisfying $\lambda(\infty)=0$, which can be extended to an infinitesimal gauge transformation $\chi$ with the required properties.

For injectivity, let $(a,\phi)\in T_{A,\Phi}C^b_{n,p}$ be such that $(a,\phi)=(d^A\chi,[\Phi,\chi])$ for some $\chi$ satisfying $\chi(\infty)=0$.  We must show that $\chi|_{S^2_\infty}=0$.  We can write $\chi|_{S^2_\infty}=i\lambda T$ and $a|_{S^2_\infty}=id\lambda T$ for some function real function $\lambda$ on $S^2_\infty$.  By assumption, $d\ast_{S^2_\infty}d\lambda=0$.  So $\lambda$ is harmonic and therefore constant.  Since $\lambda(\infty)=0$, $\lambda=0$ on $S^2_\infty$ and therefore $\chi|_{S^2_\infty}=0$ as required.

Now we prove the second statement in the proposition.  For this, we introduce some notation.  Let
\begin{align}
\tilde{G}^f_{A,\Phi}&=\big\{g\in G^u\::\:CS[A,g\cdot A]\in 8\pi^2\mathbb{Z}\big\} \\
\tilde{G}^b_n&=\big\{g\in G^u\::\:g|_{S^2_\infty}=\exp\big(\tfrac{j\pi i}{n}T\big)\text{ for some }j\in\mathbb{Z}\big\}.
\end{align}
Then the quotient of $M^b_{n,p}$ by the group $\mathbb{Z}_n\subset S^1=\R/\frac{\pi}{n}\mathbb{Z}$ is $M^b_{n,p}/\mathbb{Z}_n=C^b_{n,p}/\tilde{G}^b_n$, and $M^f_{n,p}/8\pi^2\mathbb{Z}$ is the quotient of $C_{n,p}$ by the equivalence relation $(A,\Phi)\sim g\cdot(A,\Phi)$ for $g\in \tilde{G}^f_{(A,\Phi)}$.  We claim that $\tilde{G}^b_n\subset\tilde{G}^f_{A,\Phi}$ for any $(A,\Phi)\in C_{n,p}$.  This will imply that there is a natural map
\begin{equation}\label{Mb to Mf}
M^b_{n,p}/\mathbb{Z}_n= C^b_{n,p}/\tilde{G}^b_{n}\to  M^f_{n,p}/8\pi^2\mathbb{Z}.
\end{equation}

To prove the claim, suppose that $g\in\tilde{G}^b_n$ and $(A,\Phi)\in C_{n,p}$.  Then there exists a $j\in\mathbb{Z}$ such that $g':=g\exp(-\frac{\pi j}{np}\Phi)$ satisfies $g'\in G^b$.  By Lemma \ref{lemma:chern simons additivity},
\begin{equation}
CS[A,g\cdot A] = CS[A,A'] + CS[A',g'\cdot A'].
\end{equation}
in which $A'=\exp(\tfrac{\pi j}{np}\Phi)\cdot A$.  By Lemma \ref{lemma:changing CS}, the first term on the right hand side of this equation is $-8\pi^2 j$.  The second can be evaluated using an identity,
\begin{equation}\label{equivariant degree}
\frac{1}{8\pi^2}CS[A',g'\cdot A'] = \frac{1}{24\pi^2}\int_{H^3}\tr\left(L^3 -3L\wedge(F^{A'}+(g')^{-1}F^{A'}g') \right),
\end{equation}
in which $L=(g')^{-1}(dg'+[A',g'])$.  This identity can be proved by direct calculation.  The right hand side of the identity \eqref{equivariant degree} is an integral formula for the equivariant degree of the map $g':H^3\to SU(2)$ \cite{corkharlandwinyard:2021}.  Since $g'$ is constant on $S^2_\infty$, it extends to a map from the 1-point compactification $S^3$ of $H^3$ to $SU(2)$, and its equivariant degree is an integer.  It follows that $CS[A,g\cdot A]\in 8\pi^2\mathbb{Z}$ and that $g\in \tilde{G}^f_{A,\Phi}$, as claimed.

We now show that the map \eqref{Mb to Mf} is a bijection.  For surjectivity, suppose that $(A,\Phi)\in C_{n,p}$ represents a point in $M^f_{n,p}/8\pi^2\mathbb{Z}$.  We must find a $g\in \tilde{G}^f_{A,\Phi}$ such that $g\cdot(A,\Phi)\in C^b_{n,p}$.  The boundary conditions for hyperbolic monopoles imply that $(A-A^o)|_{S^2_\infty}=i\alpha T$ for some real 1-form $\alpha$.  By solving the Poisson equation $d\ast_{S^2_\infty} d\lambda = -d\ast_{S^2_\infty}\alpha$, we can find a $\lambda:S^2_\infty\to\R$ such that $\alpha+d\lambda$ is coclosed.  We can then choose an unframed gauge transformation $g$ such that $g|_{S^2_\infty}=\exp(i\lambda T)$.  For any $s\in\R$, the gauge transformation $g_s:=g\exp(s\Phi)$ will satisfy $(g_s\cdot A-A^o)|_{S^2_\infty}=i(\alpha+d\lambda)T$ so will belong to $C^b_{n,p}$.  By Lemmas \ref{lemma:chern simons additivity} and \ref{lemma:changing CS}, we can choose $s$ so that $g_s\in \tilde{G}^f_{A,\Phi}$.

For injectivity, suppose that $(A^1,\Phi^1),(A^2,\Phi^2)\in C^b_{n,p}$ map to the same point in $M^f_{n,p}/8\pi^2\mathbb{Z}$.  Then there exists $g\in \tilde{G}^f_{A^1,\Phi^1}$ such that $(A^2,\Phi^2)=g\cdot (A^1,\Phi^1)$.  We will show that $g\in \tilde{G}^b_n$ so that $(A^1,\Phi^1)$ and $(A^2,\Phi^2)$ determine the same point in $M^b_{n,p}/\mathbb{Z}_n$.  The gauge transformation $g$ is such that $g|_{S^2_\infty}=\exp(i\lambda T)$ for some function $\lambda:S^2_\infty\to \R$.  The connections $A^1,A^2$ satisfy $(A^i-A^o)|_{S^2_\infty}=i\alpha^iT$ for some coclosed 1-forms $\alpha^1,\alpha^2$.  Since $g A^1 g^{-1}-dg g^{-1}=A^2$, $\alpha^1-\alpha^2=d\lambda$.  Since $d\ast_{S^2_\infty}\alpha^i=0$, the function $\lambda$ is harmonic and hence constant.  So we can write $g=g'\exp(\lambda\Phi^1/p)$, where $g'\in G^b$.  By Lemmas \ref{lemma:chern simons additivity} and \ref{lemma:changing CS},
\begin{equation}
-8\pi n\lambda = CS[A^1,\exp(\tfrac{\lambda}{p}\Phi^1)\cdot A^1]=CS[A^1,g\cdot A^1]-CS[A',g'\cdot A'],
\end{equation}
where $A'=\exp(\frac{\lambda}{p}\Phi^1)\cdot A^1$.  By assumption, the first term on the right hand side belongs to $8\pi^2\mathbb{Z}$, while \eqref{equivariant degree} and the following discussion show that the second term on the right hand side also belongs to $8\pi^2\mathbb{Z}$.  Therefore $\lambda\in\frac{\pi}{n}\mathbb{Z}$.  Since $\Phi ^1 |_{S^2_{\infty}}=ipT$, $g\in \tilde{G}^b_n$ as claimed.

To show that the map is a diffeomorphism, we consider the induced map on tangent vectors.  Fix $(A,\Phi)\in C^b_{n,p}$ representing a point in $M^b_{n,p}$.  Recall that
\begin{equation}
T_{A,\Phi}M^b_{n,p}=T_{A,\Phi}C^b_{n,p}/\g^b_{A,\Phi}\quad\text{and}\quad
T_{A,\Phi}M^f_{n,p}=T_{A,\Phi}C_{n,p}/\g^f_{A,\Phi},
\end{equation}
where $\g^f_{A,\Phi}$ and $\g^b_{A,\Phi}$ are given in \eqref{framedinfinitesimalgt} and \eqref{boundary IGT}.  We must show that the natural map $T_{A,\Phi}M^b_{n,p}\to T_{A,\Phi}M^f_{n,p}$ is a bijection.

For surjectivity, let $(a,\phi)\in T_{A,\Phi}C_{n,p}$.  We must find  $\chi$ such that $(a-d^A\chi,\phi-[\Phi,\chi])\in T_{A,\Phi}C^b_{n,p}$ and $\chi$ is an infinitesimal framed gauge transformation, that is, $\int_{S^2_\infty}\tr(F^A\chi)=0$.  By writing $\chi|_{S^2_\infty}=i\lambda T$, this requirement is equivalent to solving the Poisson equation $d\ast_{S^2_\infty}d\lambda = d\ast_{S^2_\infty}\alpha$ and $\int_{S^2_\infty}\tr(F^AT)\lambda=0$ on $S^2_\infty$.  The Poisson equation admits a solution, and by adding a constant this solution can be made to satisfy the integral constraint.  It can be extended to an infinitesimal gauge transformation $\chi:H^3\to \su(2)$ with the required properties.

For injectivity, let $(a,\phi)\in T_{A,\Phi}C^b_{n,p}$ be such that $(a,\phi)=(d^A\chi,[\Phi,\chi])$ for some $\chi$ satisfying $\int_{S^2_\infty}\tr(F^A\chi)=0$.  We must show that $\chi|_{S^2_\infty}=0$.  By assumption, $\chi|_{S^2_\infty}=i\lambda T$ and $a|_{S^2_\infty}=id\lambda T$, where $d\ast_{S^2_\infty}d\lambda=0$.  So $\lambda$ is harmonic and therefore constant.  Since $\int_{S^2_\infty}\tr(F^AT)\lambda=0$ and $\int_{S^2_\infty}\tr(F^AT)\neq0$ this constant is zero.  So $\chi_{S^2_\infty}=0$ as required.
\end{proof}

Having introduced the framed moduli spaces $M^b_{n,p}$ and $M^\infty_{n,p}$, we now discuss complex structures.  The moduli space $M^\infty_{n,p}$ admits a natural complex structure that is defined as follows.  Recall that a hyperbolic monopole $x=(A,\Phi)$ gives a circle-invariant instanton $B=A+\Phi\,d\theta$ on $\R^4=S^4\setminus\{\infty\}$.  We identify $\R^4$ with $\C^2$ via the mapping $(\rho,\theta,x,y)\mapsto (\rho e^{i\theta},x+iy)$.  It can be shown that $T_xM^\infty_{n,p}$ is isomorphic to the space of circle-invariant $\su(2)$-valued 1-forms $b=a+i\phi\,d\theta$ satisfying
\begin{equation}\label{infinity framed tangent space}
\ast_E d^Bb = -d^Bb,\quad d^B\ast_E b=0.
\end{equation}
Here $\ast_E$ denotes the Hodge star operator for the euclidean metric.  The first equation is equivalent to the linearised Bogomolny equation \eqref{linearised bog_eq}, and the second is a gauge-fixing condition.  The complex structure on $M^\infty_{n,p}$ is defined by
\begin{equation}\label{infinity framed J}
J_\infty b = J_\infty(b^{1,0}+b^{0,1})=-ib^{1,0}+ib^{0,1}.
\end{equation}
This complex structure is inherited from a complex structure on the moduli space of framed hermitian-Yang-Mills connections on $\C^2$.  The circle group acts holomorphically on the instanton moduli space, and $M^\infty_{n,p}$ is the fixed set of the circle action.  Therefore $M^\infty_{n,p}$ is a complex submanifold and, in particular, $J_\infty$ is integrable.

We wish to compare this with the almost complex structures $J_{\langle\psi\rangle}$ defined on $M^f_{n,p}$.  We begin by writing both almost complex structures in local coordinates.  Starting with $M^\infty_{n,p}$, equations \eqref{infinity framed tangent space} take the form
\begin{align}\label{eigenspinor explicit}
(\partial^A_x+i\partial^A_y)(a_\rho+\tfrac{i}{\rho}\phi)-(\partial^A_\rho+\tfrac{i}{\rho}\mathrm{ad}\Phi)(a_x+{i}a_y)&=0,\\
\partial^A_x a_y - \partial^A_y a_x + \tfrac{1}{\rho}\partial^A_\rho \phi-\tfrac{1}{\rho}[\Phi,a_\rho]&=0,\\
\partial^A_x a_x + \partial^A_y a_y + \partial^A_\rho a_\rho+\tfrac{1}{\rho}a_\rho+\tfrac{1}{\rho^2}[\Phi,\phi]&=0.
\end{align}
The complex structure \eqref{infinity framed J} is such that
\begin{equation}\label{J explicit}
\begin{aligned}
((J_\infty a)_x+i(J_\infty a_y))&=i((J_\infty a)_x+i(J_\infty a_y)),\\
((J_\infty a)_\rho+i\rho (J_\infty\phi))&=i((J_\infty a)_\rho+i\rho (J_\infty\phi)).
\end{aligned}
\end{equation}

Turning to $M^f_{n,p}$, we assume that our tangent vector is written in the form $\cl_H(a+\phi)=\nu\otimes\psi+\bar{\nu}\otimes\bar{\psi}$ for some $\nu\in E_x^+$, where $\psi=\bar{\psi_1}\in K^+$ and $\psi_1$ is given in \eqref{explicit Killing spinors}.  Explicitly, we have that
\begin{equation}
\begin{aligned}
\cl_S(a,\phi) &= \begin{pmatrix}\phi- i\rho a_\rho &- i\rho(a_x-ia_y) \\ -i\rho(a_x+ia_y) & \phi+ i\rho a_\rho\end{pmatrix},\,\\
\psi^\ast&=\begin{pmatrix}\rho^{-\frac12}&0\end{pmatrix},\quad
\bar{\psi}^\ast=\begin{pmatrix}0&-\rho^{-\frac12}\end{pmatrix}.
\end{aligned}
\end{equation}
It follows that
\begin{equation}
\nu=\rho^{\frac12}\begin{pmatrix}\phi- i\rho a_\rho \\ -i\rho(a_x+ia_y)\end{pmatrix}.
\end{equation}
From \eqref{clifford isomorphism} and \eqref{Levi Civita connection}, the Dirac operator is given by
\begin{equation}
D^A -\ad(\Phi) = \begin{pmatrix}-i\rho\partial^A_\rho+i-\ad(\Phi) & -i\rho\partial_x^A-\partial_y^A \\ -i\rho\partial_x^A+\partial_y^A & i\rho\partial^A_\rho-i-\ad(\Phi) \end{pmatrix}.
\end{equation}
It can then be shown that the eigenspinor equation $D^A\nu-[\Phi,\nu]=\frac{i}{2}\nu$ is equivalent to \eqref{eigenspinor explicit}.  Moreover, the action \eqref{J definition} of the almost complex structure $J_{\langle\psi\rangle}$ agrees with \eqref{J explicit}.  So it is tempting to conclude that the almost complex structure $J_{\langle\psi\rangle}$ on $M^f_{n,p}$ agrees with the almost complex structure $J$ on $M^\infty_{n,p}$.

However, this is not straightforward to prove.  In order to compare the two almost complex structures, we use the maps
\begin{equation}\label{tangent space maps}
T_xM^f_{n,p}\to T_xM^b_{n,p}\to T_xM^\infty_{n,p}
\end{equation}
defined in Proposition \ref{prop:3 framings}.  Let $(a,\phi)$ be a solution of \eqref{eigenspinor explicit} representing a tangent vector to $M^f_{n,p}$.  To map this tangent vector to $T_xM^b_{n,p}$ we must put it in the correct gauge.  We find an infinitesimal gauge transformation $\chi^f\in \g^f_{A,\Phi}$ so that $(a+d^A\chi^f,\phi+[\Phi,\chi^f])$ restricts to a harmonic 1-form on $S^2_\infty$ and hence defines a tangent vector to $M^b_{n,p}$.  It then determines a tangent vector to $M^\infty_{n,p}$ via the natural map $T_xM^b_{n,p}\to T_xM^\infty_{n,p}$.  We then find a infinitesimal gauge transformation $\chi^\infty\in \g^\infty$ such that
\begin{equation}\label{f b infty}
(a+d^A(\chi^f+\chi^\infty),\phi+[\Phi,\chi^f+\chi^\infty])
\end{equation}
once again solves the gauge-fixing condition in \eqref{eigenspinor explicit}.

The 1-form and 0-form in equation \eqref{f b infty} are not the same as the $(a,\phi)$ that we started with.  They differ by the gauge transformation $\chi^f+\chi^\infty$.  If $\chi^f(\infty)\neq0$, this gauge transformation represents a generator of the circle action on $M^b_{n,p}$.  Then $J_\infty$ is not obviously the same as $J_{\langle\psi\rangle}$, because applying \eqref{J explicit} to \eqref{f b infty} is not the same as applying \eqref{J explicit} to $(a,\phi)$.

We can summarise the situation as follows.  Let us identify $T_xM^f_{n,p}$ and $T_xM^\infty_{n,p}$ using the maps \eqref{tangent space maps}.  Let $V$ denote a generator of the circle-action and let $X=(a,\phi)\in T_x M^f_{n,p}$.  We can write $d^A(\chi^f+\chi^\infty)=\theta(X)V$, where $\theta$ is a 1-form on $M^f_{n,p}$.  We have shown that
\begin{equation}
J_\infty X - J_{\langle\psi\rangle}X=\theta(X)J_{\infty}V-\theta(J_{\langle\psi\rangle}X)V.
\end{equation}
So $J_{\langle\psi\rangle}=J_\infty$ if and only if $\theta=0$, and the integrability of $J_{\langle\psi\rangle}$ will depend on properties of the 1-form $\theta$.  Since the definition of $\theta$ is not very explicit, it is not straightforward to show that $J_{\langle\psi\rangle}$ is integrable.
%
%\begin{proposition}
%For each almost complex structure $J_{\langle \psi\rangle}$ on $M^f_{n,p}$, there is a choice of point $\infty\in S^2_\infty$ such that the map $M^f_{n,p}\to M^\infty_{n,p}$ preserves almost complex structures.
%\end{proposition}
%\begin{proof}
%It is enough to prove the proposition for one choice of $\psi\in K^+\setminus\{0\}$, because the isometry group $SL(2,\C)$ of $H^3$ acts transitively on the space $K^+\{0\}$ of nonzero Killing spinors.  We will choose $\psi=\bar{\psi}_1$, where $\psi_1$ is given in \eqref{explicit Killing spinors} in coordinates $\rho,x,y$ on the upper half space model.  The appropriate choice of point in $S^2_\infty$ will be the point at infinity in the upper half space model.

%Let $(a,\phi)$ be a tangent vector to $M^f_{n,p}$ at $x=(A,\Phi)$.  We write this in the form
%\begin{equation}
%\cl_H(a+\phi) = \nu\otimes\psi+\bar{\nu}\otimes\bar\psi
%\end{equation}
%for some $\nu\in E^+_x$, so that the almost complex structure is given by \eqref{}.  Then $\nu\otimes\psi$ and $\bar{\nu}\otimes\bar\psi$ satisfies the ``$+$'' and ``$-$'' gauge fixing conditions in \eqref{gauge fixing}.  Using Proposition \ref{}, we map this tangent vector to $T_xM^b_{n,p}$ and then to $T_xM^\infty_{n,p}$.
%
%Tangent vectors to $M^b_{n,p}$ must restrict to harmonic 1-forms on $S^2_\infty$ (see \eqref{}).  We must check whether $(a,\phi)$ fulfils this requirement.  Since $\bar{\nu}\otimes\bar\psi$ satisfies the ``$-$'' gauge-fixing condition, it 
%\end{proof}

\newpage 
\bibliographystyle{amsplain}
\bibliography{biblio}

\end{document}